\documentclass[12pt]{amsart}

\usepackage{amssymb}
\usepackage{amsmath}
\usepackage{verbatim}     

\makeatletter
\@namedef{subjclassname@2010}{%
  \textup{2010} Mathematics Subject Classification}
\makeatother

\allowdisplaybreaks

\newtheorem{thm}{Theorem}
\newtheorem{prp}{Proposition}
\newtheorem{cor}{Corollary}
\newtheorem{lem}{Lemma}

\newtheorem*{thmA}{Theorem A}

\newcommand{\dC}{\mathbb C}                             
\newcommand{\dR}{\mathbb R}                             
\newcommand{\dN}{\mathbb N}                             
\newcommand{\dZ}{\mathbb Z}                             
\newcommand{\ra}{\rightarrow}                           
\newcommand{\abs}[1]{\left| #1 \right|}                 
\newcommand{\Abs}[1]{\left\| #1 \right\|}               
\newcommand{\floor}[1]{\left\lfloor #1 \right\rfloor}   
\newcommand{\ceil}[1]{\left\lceil #1 \right\rceil}      
\newcommand{\p}[1]{\left( #1 \right)}                   
\DeclareMathOperator{\e}{\mathrm{e}}                    
\newcommand{\ud}{\,\mathrm{d}}                          

\begin{document}

\title[Piatetski-Shapiro via Beatty sequences]{Piatetski-Shapiro sequences via Beatty sequences}
\author[L. Spiegelhofer]{Lukas Spiegelhofer}
\address{Institut f\"ur Diskrete Mathematik und Geometrie,
Technische Universit\"at Wien,
Wiedner Hauptstrasse 8--10,
1040 Wien, Austria}
\email{lukas.spiegelhofer@tuwien.ac.at}
\date{}
\begin{abstract}
Integer sequences of the form $\lfloor n^c\rfloor$,
where $1<c<2$,
can be locally approximated by sequences of the form
$\lfloor n\alpha+\beta\rfloor$ in a very good way.
Following this approach, we are led to an
estimate of the difference
\[\sum_{n\leq x}\varphi\left(\lfloor n^c\rfloor\right)
-
\frac 1c\sum_{n\leq x^c}\varphi(n)n^{\frac 1c-1},\]
which measures the deviation of the mean value of $\varphi$
on the subsequence $\lfloor n^c\rfloor$ from the expected value,
by an expression involving exponential sums.
As an application we prove that for $1<c\leq 1.42$
the subsequence of the
Thue-Morse sequence indexed by $\lfloor n^c\rfloor$
attains both of its values with asymptotic density $1/2$.
\end{abstract}

\maketitle
\section{Introduction}
Piatetski-Shapiro sequences
are sequences of the form $\p{\floor{n^c}}_{n\geq 1}$, where $c>1$ is not an integer.
They are named after I. Piatetski-Shapiro, who proved the following Prime Number Theorem (see~\cite{P53}):
If $1<c<\frac {12}{11}$, then
\begin{equation}\label{eqn:PS_thm}
\abs{\left\{n\leq x: \floor{n^c}\text{ is prime}\right\}}\sim \frac{x}{c\log x}.
\end{equation}
The range for $c$ has been extended several times, the currently best known upper bound being $c<\frac{2817}{2426}$
obtained by Rivat and Sargos~\cite{RS01}.
It is expected that the asymptotic formula~(\ref{eqn:PS_thm})
holds for all $c\in (1,2)$, an expectation that is backed up by the fact that it is true for
almost all $c\in [1,2]$ with respect to the Lebesgue measure (see~\cite{LW75}).

For a collection of various arithmetic results on Piatetski-Shapiro sequences see the article
\cite{BBBSW09} by Baker et al.
For example in that article it is proved in detail that for $1<c<\frac{149}{87}$
the number of squarefree integers of the form $\floor{n^c}$ behaves as expected:
for $c$ in this range we have
\begin{equation}\label{eqn:squarefree}\nonumber
\abs{\left\{n\leq x: \floor{n^c}\text{ is squarefree}\right\}}=\frac 6{\pi^2}x+O\p{x^{1-\varepsilon}}
.
\end{equation}
According to that paper, this result was sketched by Cao and Zhai~\cite{CZ08} before.

A more basic question is to ask for the distribution of $\floor{n^c}$ in residue classes.
In this case it is known that for all noninteger $c>1$,
all positive integers $m$ and all $a\in\dZ$ we have
\begin{equation}\label{eqn:residue_classes}\nonumber
\abs{\{n\leq x:\floor{n^c}\equiv a \bmod m\}}=\frac xm+O\p{x^{1-\varepsilon}}
\end{equation}
for some $\varepsilon=\varepsilon(c)$ that can be given explicitly,
see Deshouillers~\cite{D73} and Morgenbesser~\cite{M11}.

Another line of research was initiated by Mauduit and Rivat~\cite{MR95} which concerns the behaviour of
$q$-multiplicative functions on Piatetski-Shapiro sequences.
For an integer $q\geq 2$, a function $\varphi:\dN\ra\dC$ is called $q$-\emph{multiplicative}
if for all $a\geq 0$, $k\geq 0$ and for $0\leq b<q^k$ we have
$\varphi\p{q^ka+b}=\varphi\p{q^ka}\varphi(b)$.
The function $\e\p{\alpha s_q(n)}$, where $s_q$ denotes the
sum-of-digits function in base $q$, and the trigonometric monomial
$\e\p{\alpha n}$ are examples of $q$-multiplicative functions.
Gelfond~\cite{G68} solved the problem
of describing the distribution of the values $s_q(n)$ in residue classes,
where $n$ itself is restricted to a residue class,
and posed the analogous problem of describing the distribution of
$s_q(P(n))$ in residue classes, where
$P$ is a polynomial of degree greater than one such that $P(\dN)\subseteq \dN$.
The study of $q$-multiplicative functions on Piatetski-Shapiro sequences
can be seen as a step towards the resolution of this question,
in the same way that the Piatetski-Shapiro Prime Number Theorem is
an approach to unsolved problems such as proving that there are
infinitely many prime numbers of the form $n^2+1$.
In~\cite{MR05} Mauduit and Rivat proved the following theorem.
\begin{thmA}[Mauduit and Rivat]
Let $c\in \p{1,7/5}$ and $\gamma=1/c$.
For all $\delta\in \p{0,(7-5c)/9}$
there exists a constant $C=C(\gamma,\delta)$ such that
for all $q$-multiplicative functions $\chi$ 
and all $x\geq 1$ we have
\begin{equation}\label{eqn:mr05}
\abs{
    \sum_{1\leq n\leq x}\chi\p{\floor{n^c}}
  -
    \sum_{1\leq m\leq x^c}\gamma m^{\gamma-1}\chi(m)
}
\leq
C(\gamma,\delta)
x^{1-\delta}
.
\end{equation}
\end{thmA}

Morgenbesser~\cite{M11}
gave a nontrivial bound for the sum
$\sum\e\p{\alpha s_q\p{\floor{n^c}} }$ for all noninteger $c>1$,
provided only that $q$ is large enough (depending on $c$).
Deshouillers, Drmota and Morgenbesser~\cite{DDM12} investigated
subsequences of automatic sequences of the form $\floor{n^c}$ for $c<7/5$
by generalizing the method from~\cite{MR05}.
Mauduit and Rivat~\cite{MR09} gave a complete description
of the distribution of the sum of digits of squares in residue classes,
thus solving the conjecture of Gelfond for the case that $P(X)=X^2$.
The problem of proving~(\ref{eqn:mr05}) for the case that $c\geq 7/5$ is not an integer,
$\chi(n)=\e\p{\alpha s_q(n)}$ and $q$ is small could not be solved, however.

In the present article we follow a new approach to
problems on Piatetski-Shapiro sequences.
This approach is based on the idea of approximating the function $x^c$
by a family of tangents $x\alpha+\beta$, each restricted to a small interval.
Let $\delta\in (0,1-c/2)$ and $\varepsilon>0$ be given.
Then by linear approximation we can choose for $x_0\geq 1$ some
$\alpha$ and $\beta$ in such a way that
$\abs{x^c-x\alpha-\beta}<\varepsilon$ if $\abs{x-x_0}<Cx^\delta$,
where $C$ does not depend on $x_0$.
It seems therefore likely that $\floor{n^c}=\floor{n\alpha+\beta}$
for most integers $n$ in such an interval.
These observations are made precise by the lemmas in Section~\ref{subsection_proof_prp_1}.

Algebraic properties of the function $x\mapsto x^c$ are not needed for such an approximation.
Correspondingly our method can be adapted to treat functions from a larger class,
defined by certain conditions on the derivatives.
Functions like $x^c\log^\eta x$ or $x^c\exp\p{\log^\varepsilon x}$,
where $1<c<2$, $\eta\in\dR$ and $0\leq\varepsilon<1$,
are contained in this class as well as
linear combinations with positive coefficients of its elements.

A sequence of integers of the form $\p{\floor{n\alpha+\beta}}_{n\geq 1}$, where $\alpha>0$, is called a
\emph{(non-homogeneous) Beatty sequence}.
They are named after S. Beatty, who posed a problem (concerning the homogeneous case)
in the American Mathematical Monthly in 1926 (see~\cite{B26}), which essentially states that
for irrational $\alpha_1, \alpha_2>1$ such that $\frac 1{\alpha_1}+\frac 1{\alpha_2}=1$ the sequences
$\p{\floor{n\alpha_1}}_{n\geq 1}$ and $\p{\floor{n\alpha_2}}_{n\geq 1}$ form a partition of the set of positive integers.
This fact was already found in 1894 by Rayleigh~\cite[pp.122--123]{R1894} and correspondingly it is called
Rayleigh's Theorem or Beatty's Theorem.
We refer to~\cite{BS11} for some references to the newer literature concerning Beatty sequences.

We consider a bounded arithmetic function $\varphi$ and a differentiable function
$f:\dR^+\ra\dR^+$ satisfying $f'>0$ and other conditions on its derivatives and ask whether it is true that
\begin{equation}\label{eqn:substitution_rule}
        \sum_{A<n\leq 2A}{
          \varphi\p{\floor{f(n)}}
        }
      -
        \sum_{f(A)<m\leq f(2A)}{
          \varphi(m)\p{f^{-1}}'(m)
       }
    =
      o(A)
\end{equation}
as $A\ra\infty$.
The two terms on the left hand side resemble the terms involved in the change of variables in an integral.
Heuristically, we expect therefore that ``well behaved'' functions $\varphi$ yield a small error term on the right hand side.
This expectation is in general very difficult to verify,
which is obvious from the observation that,
for instance, (\ref{eqn:PS_thm})
can be reduced to a statement of the form~(\ref{eqn:substitution_rule}).

The main result of this paper, based on the method of approximating $\floor{n^c}$ by Beatty sequences
and the approximation of the periodic Bernoulli polynomial $\psi(x)=x-\floor{x}-\frac 12$ by trigonometric polynomials,
is a sufficient condition for the statement~(\ref{eqn:substitution_rule}) to hold.
More precisely we give an upper bound on the error term that involves the exponential sum
$\sum \varphi(m)e(m\theta)$ over short intervals.

We give several application of this theorem.
The first application is an improvement of the bound $7/5=1.4$ in Theorem A
to the value $1.42$
in the case that $\chi$ is the Thue-Morse sequence,
which expresses the parity of the number of ones in the binary representation of a natural number.
In order to prove this result, we use an estimate of the $L^1$-norm of the corresponding exponential sum (as a function in $\theta$) given by Fouvry and Mauduit~\cite{FM96}.

Another application concerns the joint distribution of sum-of-digits functions
on Piatetski-Shapiro sequences.
It is another problem posed in the paper~\cite{G68} by Gelfond to prove that
if $q_1,q_2\geq 2$, $m_1,m_2\geq 1$ and $l_1,l_2$ are integers
such that $(q_1,q_2)=1$, $(m_1,q_1-1)=1$ and $(m_2,q_2-1)=1$,
there exists $\varepsilon>0$ such that
\begin{multline}\label{eqn:joint_introduction}
    \abs{
      \{n\leq x:
        s_{q_1}(n)\equiv l_1\bmod m_1\text{ and }
        s_{q_2}(n)\equiv l_2\bmod m_2
      \}
    }
\\
  =
    \frac{x}{m_1m_2}
  +
    O\p{x^{1-\varepsilon}}
.
\end{multline}
This statement was proved by Kim~\cite{K99}, but a weaker form of this result,
specifically with a non-explicit error term,
was provided by B\'esineau long before (see~\cite{B72}).
To the author's knowledge the problem
of proving a result such as~(\ref{eqn:joint_introduction})
for subsequences $\floor{n^c}$ of the integers
has not been dealt with in the literature before.
We obtain such a result for all $c$ in the interval $(1,18/17)$.
In the proof we make (besides the main theorem)
use of discrete Fourier coefficients related to the sum-of-digits function.
These Fourier coefficients have proven to be an excellent tool
for treating problems related to the sum of digits (see~\cite{MR09,MR10})
and can also be used in this context.
We also note that their use leads to an alternative method of proving~(\ref{eqn:joint_introduction}).

As the third application we prove a result on the distribution in residue classes
of the Zeckendorf sum-of-digits function $s_Z$ evaluated on Piatetski-Shapiro sequences.
By the well-known theorem of Zeckendorf~\cite{Z72} every positive integer $n$ can be represented uniquely as a sum of non-consecutive Fibonacci numbers.
The number of summands in this representation is called the \emph{Zeckendorf sum-of-digits} of $n$, which we denote by $s_Z(n)$.
We prove that for integers $m\geq 1$ and $a$
and for all $c\in (1,4/3)$
there exists $\varepsilon>0$ such that
\begin{equation}\label{eqn:zeckendorf_residue_classes_introduction}\nonumber
    \abs{\left\{n\leq x: s_Z\p{\floor{n^c}}\equiv a\bmod m\right\}}
  =
    \frac xm
  +
    O\p{x^{1-\varepsilon}}
.
\end{equation}
In this article, we denote the set of positive real numbers by $\dR^+$
and the set of nonnegative integers by $\dN$.
For $x\in\dR$ we write $\e(x)=\e^{2\pi ix}$, $\Abs{x}=\min_{n\in\dZ}\abs{n-x}$ and $\{x\}=x-\floor{x}$.
Conditions like $i<n$ under a summation or product sign are to be read as $0\leq i<n$.
\section{Main results}
The main result is an estimate of the error term in~(\ref{eqn:substitution_rule}) for a special class of functions $f$.
\begin{thm}\label{thm_1}
Assume that $f$ is a two times continuously differentiable real valued function on $\dR^+$
such that
$f,f',f''>0$
and that there exist $c_1\geq 1/2$ and $c_2>0$
such that for $0<x\leq y\leq 2x$ we have
$c_1f''(x)\leq f''(y)\leq c_2f''(x)$.
Let $A_0\geq 2$ be such that $f'(A_0)\geq 1$.
There exists a constant $C=C(f)$ such that
for all complex valued arithmetic functions $\varphi$ bounded by $1$,
 for all integers $A\geq A_0$ and for all $z>0$ we have
\begin{multline}\label{eqn:expsum_conclusion}
    \frac 1A
    \abs{
        \sum_{A<n\leq 2A}{
          \varphi\p{\floor{f(n)}}
        }
      -
        \sum_{f(A)<m\leq f(2A)}{
          \varphi(m)\p{f^{-1}}'(m)
        }
    }
\\
  \leq
    C
    \p{
        \frac{f''(A)}{f'(A)^2}z^2
      +
        f'(A)(\log A)^3 J(A,z)
    }
,
\end{multline}
where
\begin{equation}\label{eqn:expsum_integral}
    J(A,z)
  =
    \int_0^1{
      \sup_{
        f(A)<x\leq f(2A)
      }
      \frac 1z
      \abs{
        \sum_{x<m\leq x+z}{
          \varphi(m)\e\p{m\theta}
        }
      }
    }
    \ud \theta
.
\end{equation}
\end{thm}
Theorem~\ref{thm_1} is a consequence of the following result,
which provides a way to prove a discrete substitution rule
by solving a problem about the behaviour of $\varphi$ on Beatty sequences.
\begin{prp}\label{prp_1}
Assume that $f$ is a two times continuously differentiable real valued function on $\dR^+$
such that
$f,f',f''>0$,
and that there exist $c_1\geq 1/2$ and $c_2>0$
such that for $0<x\leq y\leq 2x$ we have
$c_1f''(x)\leq f''(y)\leq c_2f''(x)$.
There exists $C=C(f)$ such that
for all complex valued arithmetic functions $\varphi$ bounded by $1$, for
all $A\geq 2$ and $K>0$
we have
\begin{multline}\label{eqn:essential_conclusion}
    \frac 1A
    \abs{
        \sum_{A<n\leq 2A}{
          \varphi\p{\floor{f(n)}}
        }
      -
        \sum_{f(A)<m\leq f(2A)}{
          \varphi(m)\p{f^{-1}}'(m)
        }
    }
\\
  \leq
    C
    \p{
        f''(A)K^2
      +
        \frac{(\log A)^2}{K}
      +
        I(A,K)
    }
,
\end{multline}
where $I(A,K)$ is defined by
\begin{multline}\label{eqn:essential_integral}
    I(A,K)
  =
    \frac 1{f'(2A)-f'(A)}
\\
  \times
    \int_{f'(A)}^{f'(2A)}{
      \sup_{
          f(A)<\beta\leq f(2A)
      }
      \frac 1K
      \abs{
          \sum_{0<n\leq K}{
            \varphi\p{\floor{n\alpha+\beta}}
          }
        -
          \frac 1\alpha
          \sum_{\beta<m\leq \beta+K\alpha}{
            \varphi(m)
          }
      }
    }
    \ud \alpha
.
\end{multline}
\end{prp}
\section{Applications}
In the proofs of our applications, concerning sum-of-digits functions, we make use of bounds for the exponential sum $\sum_{x<m\leq x+z}\varphi(m)\e\p{m\theta}$
that are independent of the value of $x$.
Moreover, for simplicity we concentrate on the case that $f(x)=x^c$, although
it would be possible to derive analogous results for a larger class of functions,
as we noted in the introduction.
We state a corollary of Theorem~\ref{thm_1} that is adjusted to this situation.
\begin{cor}\label{cor_1}
Let $\varphi$ be a complex valued arithmetic function bounded by $1$.
If
$a\in(0,1]$
and $C$
are such that
\begin{equation}\label{eqn:cor_expsum_integral}
    \int_0^1{
      \sup_{
        x\geq 0
      }{
        \abs{
          \sum_{x<m\leq x+z}{
            \varphi(m)\e\p{m\theta}
          }
        }
      }
    }
    \ud \theta
  \leq
    Cz^a
\end{equation}
for $z\geq 1$, then
for all $c\in (1,2)$ and all $\eta\in\p{0,\frac{2-(a+1)c}{3-a}}$
there is a $C_1=C_1(a,c,C,\eta)$ such that
\begin{equation}\label{eqn:cor_expsum_conclusion}
    \frac 1N
    \abs{
        \sum_{1\leq n\leq N}{
          \varphi\p{\floor{n^c}}
        }
      -
        \frac 1c
        \sum_{1\leq m\leq N^c}{
          \varphi(m)m^{\frac 1c-1}
        }
    }
  \leq
    C_1N^{-\eta}
\end{equation}
for $N\geq 1$.
\end{cor}
\begin{proof}
For $A>0$ we write
\begin{equation}\label{eqn:dfn_F}
    F(A)
  =
    \abs{
        \sum_{A<n\leq 2A}{
          \varphi\p{\floor{n^c}}
        }
      -
        \frac 1c
        \sum_{A^c<m\leq (2A)^c }{
          \varphi(m)m^{\frac 1c-1}
        }
    }
.
\end{equation}
Let $1<c<2$ and set
$z=A^\frac{2c-1}{3-a}$ for $A\geq 2$.
From
hypothesis~(\ref{eqn:cor_expsum_integral})
and Theorem~\ref{thm_1} it follows
by a short calculation that
for all integers $A\geq 2$ and all $\varepsilon>0$ we have
\begin{equation}\label{eqn:errorterm_rho}
    F(A)
  \ll
    A^{1-\rho+\varepsilon}
\end{equation}
with the choice $\rho=\frac{2-c(a+1)}{3-a}$.
The implied constant in~(\ref{eqn:errorterm_rho}) may depend on $a, c, C$ and $\varepsilon$.
Altering the summation limits in~(\ref{eqn:dfn_F}) to
$\floor{A}<n\leq \floor{2A}$ and $\floor{A}^c<m\leq \floor{2A}^c$ respectively
introduces an error term of $O(1)$, which is neglegible.
Therefore~(\ref{eqn:errorterm_rho}) holds for all real $A\geq 2$ and $\varepsilon>0$.
We have $F(A)=0$ for $A<\frac 12$, and it is clear that $F(A)$ is bounded for $0<A\leq 2$.
From these observations and~(\ref{eqn:errorterm_rho}) it follows that $F(A)\ll A^{1-\rho+\varepsilon}$
for all $A>0$.
Since $\rho-\varepsilon<1$ we
get
\begin{multline*}
    \abs{
        \sum_{1\leq n\leq N}{
          \varphi\p{\floor{n^c}}
        }
      -
        \frac 1c
        \sum_{1\leq m\leq N^c}{
          \varphi(m)m^{\frac 1c-1}
        }
    }
\\
  =
    \abs{
      \sum_{i\geq 1}\p{
          \sum_{
            N/{2^i}<n\leq N/{2^{i-1}}
          }{
            \varphi\p{\floor{n^c}}
          }
        -
        \frac 1c
          \sum_{
            \p{N/{2^i}}^c<n\leq\p{N/{2^{i-1}} }^c
          }{
            \varphi(m)m^{\frac 1c-1}
          }
      }
    }
\\
  \leq
    \sum_{i\geq 1}F\p{\frac N{2^i}}
  \ll
    CN^{1-\rho+\varepsilon}
.
\end{multline*}
From this the assertion follows.
\end{proof}
\subsection{The Thue-Morse sequence}
In our first application
we are interested in the special case that the function $\varphi$
is the Thue-Morse sequence in the form $\varphi(n)=(-1)^{s_2(n)}$,
where $s_2(n)$ denotes the sum of digits of $n$ in base $2$.
\begin{thm}[The Thue-Morse sequence on $\floor{n^c}$]\label{app_1}
There exists
$a\in [0,0.4076)$
such that
for all $c\in (1,2)$ and all $\eta\in\p{0,\frac{2-(a+1)c}{3-a}}$
there is a constant $C=C(c,\eta)$ such that for all $N\geq 2$
\begin{equation}\label{eqn:nc_estimate}\nonumber
    \frac 1N
    \abs{
      \sum_{1\leq n\leq N}{
        (-1)^{s_2\p{\floor{n^c}} }
      }
    }
  \leq
    CN^{-\eta}
.
\end{equation}
In particular, for $1<c\leq 1.42$ there exist
$\eta>\max\left\{0,(7-5c)/9\right\}$ and $C$ such that this estimate holds.
\end{thm}
In order to prove this, we want to apply Corollary~\ref{cor_1} and therefore we have to find
an estimate for the expression on the left hand side of~(\ref{eqn:cor_expsum_integral}).
We use the following statement which
follows from Th\'eor\`eme 3 and inequality (1.5) in the paper~\cite{FM96} by Fouvry and Mauduit.
\begin{lem}\label{lem_FM}
There exists a real number $\rho\in(0.6543,0.6632)$
such that
\[
    \int_0^1{
      \prod_{0\leq k<\lambda}{
        \abs{\sin\p{2^k\pi \theta}}
      }
    }
    \ud \theta
  \asymp
    \rho^\lambda
\]
for all $\lambda\geq 0$.
\end{lem}
The number $\rho$ is clearly uniquely determined.
No simple representation of $\rho$ seems to be known
and in fact the above bounds were obtained with the help of numerical computations.
The authors of the cited article also remark that evaluating the numerical value of the integral for
about a dozen values of $\lambda$
(by means of splitting up the interval $[0,1]$ into $2^{\lambda}$ subintervals of equal length
and using the fact that for $k<\lambda$ the function $\sin\p{2^k\pi \theta}$ has a constant sign on each of them)
suggests that $\rho=0.661\ldots$.
From Lemma~\ref{lem_FM} we deduce the following estimate, which is the main component of the proof of Theorem~\ref{app_1}.
\begin{prp}\label{prp_app_1}
Let $\rho$ be defined as in Lemma~\ref{lem_FM}.
Then uniformly for $z\geq 1$ we have
\begin{equation}\label{eqn:app_1}\nonumber
\int_0^1 \sup_{x\geq 0} \abs{ \sum_{x<m\leq x+z} (-1)^{s_2(m)}\e\p{m\theta} }\ud \theta\ll z^{1+\frac{\log \rho}{\log 2}}.
\end{equation}
\end{prp}
\begin{proof}
If $L$ is an interval of the form
$[\ell2^\lambda,(\ell+1)2^\lambda)$,
where $\ell$ and $\lambda$ are nonnegative integers, we have
the equality
\begin{equation}\label{eqn:expsum_product}
    \abs{\sum_{m\in L}(-1)^{s_2(m)}\e(m\theta)}
  =
    \prod_{0\leq k<\lambda}\abs{1-\e\p{2^k \theta}}
.
\end{equation}
This is clear for $\lambda=0$.
If $\lambda>0$, then by the relations $s_2(2m)=s_2(m)$ and $s_2(2m+1)=s_2(2m)+1$ we have
\begin{multline*}
    \abs{\sum_{m\in L}(-1)^{s_2(m)}\e(m\theta)}
\\
  =
    \abs{
      \sum_{\ell2^{\lambda-1}\leq m<(\ell+1)2^{\lambda-1}}{
        \p{
            (-1)^{s_2(2m)}\e(2m\theta)
          +
            (-1)^{s_2(2m+1)}\e((2m+1)\theta)
        }
      }
    }
\\
  =
    \abs{(1-e(\theta))}
    \abs{
      \sum_{\ell2^{\lambda-1}\leq m<(\ell+1)2^{\lambda-1}}{
        (-1)^{s_2(m)}\e(2m\theta)
      }
    }
\end{multline*}
from which~(\ref{eqn:expsum_product}) follows by induction.
Using the trigonometric identity $\abs{1-\e(\theta)}=2\abs{\sin(\pi \theta)}$ we get
\begin{equation}\label{eqn:sinprod}
    \abs{\sum_{m\in L}(-1)^{s_2(m)}\e(m\theta)}
  =
    2^\lambda\prod_{0\leq k<\lambda}\abs{\sin\p{2^k\pi \theta}}.
\end{equation}
If $L$ is any finite nonempty interval of nonnegative integers, we use dyadic decomposition of $L$
in the form of the following statement:
Let $a<b$ be nonnegative integers.
There exists a decomposition
$a=a_0\leq\ldots\leq a_L=b_L\leq\ldots\leq b_0=b$ such that
for $j<L$ we have
$a_{j+1}-a_j\in\{0,2^j\}$, $2^j\mid a_j$ and $b_j-b_{j+1}\in\{0,2^j\}$ and $2^j\mid b_j$.

To prove this, one first establishes the special case that $a<2^K\leq b<2^{K+1}$ for some $K$ and
obtains the general case by adding a multiple of $2^{K+1}$.
We skip the details of the proof since we will return to a very similar problem in Section~\ref{subsection_Zeckendorf}.
We can therefore decompose $L$ into intervals of the form
$[\ell2^\lambda,(\ell+1)2^\lambda)$
in such a way that
for each $\lambda$ there are at most $2$ such intervals of length $2^\lambda$.
From this we obtain, using~(\ref{eqn:sinprod}), that
\begin{equation}\label{eqn:monotone_bound}\nonumber
    \abs{\sum_{m\in L}(-1)^{s_2(m)}\e(m\theta)}
  \ll
    \sum_{0\leq \lambda\leq \frac{\log\abs{L}}{\log 2}}2^\lambda
    \prod_{0\leq k<\lambda}
    \abs{\sin\p{2^k\pi \theta}}.
\end{equation}
By Lemma~\ref{lem_FM} (note that in particular $2\rho>1$) this implies
\begin{multline}\label{eqn:integral_to_exponent}\nonumber
    \int_0^1{
      \sup_{
          x\geq 0
      }{
        \abs{
          \sum_{x<m\leq x+z}{
            (-1)^{s_2(m)}\e\p{m\theta}
          }
        }
      }
    }
    \ud \theta
  \ll
    \sum_{\lambda\leq \frac{\log (z+1)}{\log 2}}
    2^\lambda
    \int_0^1{
      \prod_{k<\lambda}
      \abs{\sin\p{2^k\pi \theta}}
    }
    \ud \theta
\\
  \ll
    \sum_{\lambda\leq \frac{\log (z+1)}{\log 2}}{
      2^\lambda
      \rho^\lambda
    }
  \ll
    (2\rho)^{\frac{\log (z+1)}{\log 2}+1}
  \ll
    (2\rho)^{\frac{\log z}{\log 2}}
  =
    z^{1+\frac{\log\rho}{\log 2}}
\end{multline}
for all $z\geq 1$.
\end{proof}
\begin{proof}[Proof of Theorem~\ref{app_1}]
Note first that
$1+\frac{\log\rho}{\log 2}<0.4076$
according to the estimate $\rho<0.6632$.
Combining Proposition~\ref{prp_app_1} and Corollary~\ref{cor_1} we get the following statement:
there exists $a<0.4076$ such that
for all $c\in (1,2)$ and all $\eta\in\p{0,\frac{2-(a+1)c}{3-a}}$
there exists $C$ such that
for all $N\geq 2$ we have
\begin{equation}\label{eqn:nc_substitution_rule}
    \frac 1N
    \abs{
        \sum_{1\leq n\leq N}{
          (-1)^{s_2\p{\floor{n^c}} }
        }
      -
        \frac 1c
        \sum_{1\leq m\leq N^c}{
          (-1)^{s_2(m)}m^{\frac 1c-1}
        }
    }
  \leq
    C
    N^{-\eta}
.
\end{equation}
To prove the main statement, it remains to eliminate the second sum
in this inequality.
For all nonnegative integers $K$ we have
$\sum_{m<2K}(-1)^{s_2(m)}=0$,
therefore it follows by partial summation that
\[
    \frac 1N
    \sum_{1\leq m\leq N^c}{
      (-1)^{s_2(m)}m^{\frac 1c-1}
    }
  \ll
    \frac 1N
    \p{N^c}^{\frac 1c-1}
    \sup_{1\leq u\leq N^c}{
      \abs{
        \sum_{1\leq m\leq u}{
          (-1)^{s_2(m)}
        }
      }
    }
  \ll
    N^{-c}
.
\]
This quantity is dominated by the error term, so we may remove the second sum in~(\ref{eqn:nc_substitution_rule}).
To finish the proof, we note
that $2-(a+1)c>0$
and $\frac{7-5c}{9}<\frac{2-(a+1)c}{3-a}$
for $c\leq 1.42$ and $a<0.4076$.
\end{proof}
We remark that our method even yields a value around $1.425$ for the upper bound on $c$,
if indeed $\rho$ is around $0.661$ as the computations suggest.
In~\cite[p.579]{FM96}, an analogous remark on the dependence of a parameter on $\rho$ is made.
\subsection{The joint distribution of sum-of-digits functions}
For integers $q\geq 2$ and $n\geq 0$ we denote by $s_q(n)$ the sum-of-digits of $n$ in base $q$.
In this section we prove the following independence result of sum-of-digits functions
with respect to coprime bases $q_1$ and $q_2$.
\begin{thm}[Joint distribution of sum-of-digits functions on $\floor{n^c}$]\label{app_2}
Let $q_1,q_2\geq 2$, $m_1,m_2,\geq 1$ and $l_1,l_2$ be integers
such that $(q_1,q_2)=1$, $(m_1,q_1-1)=1$ and $(m_2,q_2-1)=1$.
Let $1<c<18/17$. There exists $\varepsilon>0$ such that
\begin{multline}\label{eqn:joint}
    \abs{
      \{n\leq x:
        s_{q_1}\p{\floor{n^c}}\equiv l_1\bmod m_1
        \text{ and }
        s_{q_2}\p{\floor{n^c}}\equiv l_2\bmod m_2
      \}
    }
\\
  =
    \frac{x}{m_1m_2}
  +
    O\p{x^{1-\varepsilon}}
.
\end{multline}
\end{thm}
Generalizing this theorem (and its proof) to more than two bases is straightforward,
however the upper bound on $c$ that we can obtain using our method has then to be adjusted.
In order to prove Theorem~\ref{app_2}, we estimate the relevant integral
as well as the integrand at $\theta=0$.
\begin{prp}\label{prp_independence_exponential_sum}
Let $q_1,q_2\geq 2$ be relatively prime integers.
There exists $C=C(q_1,q_2)$ such that for all $\alpha,\beta\in\dR$ and $z\geq 1$ we have
\begin{equation}\label{eqn:independence_exponential_sum}
    \int_0^1{
      \sup_{x\geq 0}{
        \abs{
          \sum_{x<n\leq x+z}{
            \e\p{\alpha s_{q_1}(n)+\beta s_{q_2}(n)+n\theta}
          }
        }
      }
    }
    \ud\theta
  \leq
    C
    z^{8/9}
.
\end{equation}
Moreover, we have
\begin{equation}\label{eqn:independence_sup_bound}
  \sup_{x\geq 0}{
    \abs{
      \sum_{x<n\leq x+z}{
        \e\p{\alpha s_{q_1}(n)+\beta s_{q_2}(n)}
      }
    }
  }
\leq
  C_1z^{1-\eta(\alpha)}
\end{equation}
for $z\geq 1$,
where $\eta(\alpha)=\frac{\Abs{(q_1-1)\alpha}^2}{15\log q_1}$
and $C_1$ may depend on $\alpha,\beta,q_1$ and $q_2$.
\end{prp}
In the proof of this proposition we make use of the truncated sum-of-digits function
$s_{q,\lambda}$, which adds up the first $\lambda$ digits
of the base-$q$ representation of a nonnegative integer $n$.
That is,
if $n=\sum_{i\geq 0}\varepsilon_iq^i$ and $\varepsilon_i\in\{0,\ldots,q-1\}$ for all $i$, then
\[
    s_{q,\lambda}(n)
  =
    \sum_{0\leq i<\lambda}\varepsilon_i
  =
    s_q\p{n\bmod q^\lambda}
.
\]
For convenience we extend $s_{q,\lambda}$ to a $q^\lambda$-periodic function on $\dZ$.
By periodicity, we can represent the function $\e\p{\alpha s_{q,\lambda}(n)}$
with the aid of the discrete Fourier transform.
For integers $q\geq 2$, $\lambda\geq 0$ and $n$ we have
\begin{equation}\label{eqn:fourier_transform}
\e\p{\alpha s_{q,\lambda}(n)}
=
\sum_{h<q^\lambda}\e\p{hnq^{-\lambda}}F_{q,\lambda}(h,\alpha)
\end{equation}
and
\begin{equation}\label{eqn:fourier_transform2}
\e\p{-\alpha s_{q,\lambda}(n)}
=
\sum_{h<q^\lambda}\e\p{hnq^{-\lambda}}\overline{F_{q,\lambda}(-h,\alpha)}
,
\end{equation}
where
\begin{equation*}
F_{q,\lambda}(h,\alpha)
=
\frac 1{q^\lambda}\sum_{u<q^\lambda}\e\p{\alpha s_{q,\lambda}(u)-huq^{-\lambda}}
.
\end{equation*}
The Fourier coefficients $F_{q,\lambda}(h,\alpha)$ may be estimated uniformly in $h$
using the following lemma~(\cite[Lemme 9]{MR09}).
\begin{lem}\label{lem_fourier_coeff_estimate}
Let $q,\lambda \geq 2$ and $h$ be integers and $\alpha\in\dR$. Then
\begin{equation}\label{eqn:fourier_coeff_estimate}\nonumber
\abs{F_{q,\lambda}(h,\alpha)}
\leq
\e^{\pi^2/48}q^{-c_q\Abs{(q-1)\alpha}^2\lambda}
,
\end{equation}
where
\begin{equation}\label{eqn:fourier_coeff_exponent}\nonumber
c_q=\frac{\pi^2}{12\log q}\p{1-\frac{2}{q+1}}
.
\end{equation}
\end{lem}
We prove the following lemma on the truncated sum-of-digits function,
which is a way of expressing the idea that addition of an integer $r$ to $n$
should only change digits at low positions in most cases.
\begin{lem}\label{lem_f_lambda}
Let $q\geq 2$, $\lambda\geq 0$ and $r$ be integers and
let $I$ be a finite interval in $\dN$ such that $I+r\subseteq \dN$.
Then
\[
  \abs{\{n\in I:s_q(n+r)-s_q(n)\neq s_{q,\lambda}(n+r)-s_{q,\lambda}(n)\}}
\leq
  \abs{I}\frac{\abs{r}}{q^\lambda}+\abs{r}
.
\]
\end{lem}
\begin{proof}
It is sufficient to assume that $r$ is nonnegative,
since the other case then follows by shifting the interval $I$.

For a nonnegative integer $n$, there exist unique $t$ and $u$ such that $n=tq^\lambda+u$,
where $u<q^\lambda$. Clearly we have $s_q(n)=s_q(t)+s_q(u)$ and $s_{q,\lambda}(n)=s_q(u)$.
If $n\equiv k \bmod q^\lambda$ for some $k$ such that $0\leq k<q^\lambda-r$,
then $s_q(n+r)=s_q(t)+s_q(u+r)$ and $s_{q,\lambda}(n+r)=s_q(u+r)$,
therefore $s_q(n+r)-s_q(s)=s_{q,\lambda}(n+r)-s_{q,\lambda}(n)$.
It remains therefore to show that
\hfill\\
$\abs{\{n\in I:
q^\lambda-r\leq n\bmod q^\lambda<q^\lambda
\}}$ $\leq \abs{I}r/q^\lambda+r$,
which is not difficult.
\end{proof}
The inequality of van der Corput is well known.
For our purposes, we will employ it in the following form.
\begin{lem}\label{lem_vdc}
Let $I$ be a finite interval in $\dZ$ and
let $a_n\in\dC$ for $n\in I$.
Then
\[
\abs{\sum_{n\in I}a_n}^2\leq
\frac{\abs{I}-1+R}R\sum_{0\leq\abs{r}<R}\p{1-\frac{\abs{r}}R}
\sum_{\substack{n\in I\\n+r\in I}}a_{n+r}\overline{a_n}
\]
for all integers $R\geq 1$.
\end{lem}
\begin{proof}[Proof of Proposition~\ref{prp_independence_exponential_sum}]
To estimate the left hand side of~(\ref{eqn:independence_exponential_sum}),
we introduce two parameters to be chosen later, $\lambda_1$ and $\lambda_2$.
Rounding off $z$ to the nearest multiple $M$ of $q_1^{\lambda_1}q_2^{\lambda_2}$
introduces an error term $O\p{q_1^{\lambda_1}q_2^{\lambda_2}}$.
Let $x\geq 0$, $z\geq 1$ and let $R\in [1,z]$ be an integer.
Then by van der Corput's inequality we get
\begin{multline*}
  \abs{
    \sum_{x<n\leq x+M}{
      \e\p{\alpha s_{q_1}(n)+\beta s_{q_2}(n)+n\theta}
    }
  }^2
\ll
  \frac zR
  \sum_{\abs{r}<R}\p{1-\frac{\abs{r}}R}
\\
  \times
  \sum_{x<n,n+r\leq x+M}
  \e\bigl(
      \alpha\p{s_{q_1}(n+r)-s_{q_1}(n)}
    +
      \beta\p{s_{q_2}(n+r)-s_{q_2}(n)}
    +
      r\theta
  \bigr)
.
\end{multline*}
Applying Lemma~\ref{lem_f_lambda} in order to
replace $s_{q_1}$ and $s_{q_2}$
by $s_{q_1,\lambda_1}$ and $s_{q_2,\lambda_2}$ respectively
and omitting the summation condition $x<n+r\leq x+M$ afterwards
we get an error term
$O\p{zR+z^2R\p{1/q_1^{\lambda_1}+1/q_2^{\lambda_2}} }$ and
after inserting equations~(\ref{eqn:fourier_transform}) and~(\ref{eqn:fourier_transform2}) it remains to estimate the quantity
\begin{multline}\label{eqn:independence_big_estimate}
  \frac z{R^2}
  \sum_{\substack{h_1,k_1<q_1^{\lambda_1}\\h_2,k_2<q_2^{\lambda_2} }}
  F_{q_1,\lambda_1}(h_1,\alpha)\overline{F_{q_1,\lambda_1}(-k_1,\alpha)}
  F_{q_2,\lambda_2}(h_2,\beta)\overline{F_{q_2,\lambda_2}(-k_2,\beta)}
\\
  \times
  \!\!\!
  \sum_{x<n\leq x+M}
  \!\!\!
  \e\p{n
    \!
    \p{\frac{h_1+k_1}{q_1^{\lambda_1}}+\frac{h_2+k_2}{q_2^{\lambda_2}} }
    \!
  }
  \sum_{\abs{r}<R}\p{R-\abs{r}}\e\p{r
  \!
  \p{\frac{h_1}{q_1^{\lambda_1}}+\frac{h_2}{q_2^{\lambda_2}}+\theta}
  \!
  }
.
\end{multline}
By our choice of $M$ and by the Chinese Remainder Theorem,
the contribution of the case that
$\p{h_1+k_1,h_2+k_2}\not\equiv (0,0)\bmod \p{q_1^{\lambda_1},q_2^{\lambda_2}}$ is $0$.
Using the identity
\[
  \sum_{\abs{r}<R}(R-\abs{r})\e(rx)
=
  \p{\sum_{r<R}\e(rx)}^2
,
\]
we see that~(\ref{eqn:independence_big_estimate}) is bounded by the expression
\begin{multline}\label{eqn:independence_first_part}
  \!\!
  \frac{z^2}{R^2}
  \!\!
  \sum_{
    \substack{
        h_1<q_1^{\lambda_1}
      \\
        h_2<q_2^{\lambda_2}
    }
  }
  \!\!
  \abs{F_{q_1,\lambda_1}(h_1,\alpha)}^2
  \abs{F_{q_2,\lambda_2}(h_2,\beta)}^2
  \abs{
    \sum_{\abs{r}<R}{
      \e\p{r\p{\frac{h_1}{q_1^{\lambda_1}}+\frac{h_2}{q_2^{\lambda_2}}+\theta}}
    }
  }^2
,
\end{multline}
which is independent of $x$.
In order to prove the first part of Proposition \ref{prp_independence_exponential_sum},
we use the Cauchy-Schwarz inequality,
Parseval's identity and
the identity
\[
\int_0^1\abs{\sum_{r\in I}\e\p{r(t+\theta)}}^2\ud\theta
=
\abs{I}
\]
and collect the error terms to arrive at the estimate
\begin{multline}
    \int_0^1{
      \sup_{x\geq 0}{
        \abs{
          \sum_{x<n\leq x+z}{
            \e\p{\alpha s_{q_1}(n)+\beta s_{q_2}(n)+n\theta}
          }
        }
      }
    }
    \ud\theta
\\
  =
    O\p{
        q_1^{\lambda_1}q_2^{\lambda_2}
      +
        z^{1/2}R^{1/2}
      +
        zR^{1/2}\p{q_1^{-\lambda_1/2}+q_2^{-\lambda_2/2}}
      +
        zR^{-1/2}
    }
,
\end{multline}
which is valid for all real $\alpha,\beta$ and $z\geq 1$
and all integers $R\in [1,z]$ and $\lambda_1,\lambda_2\geq 0$.
The implied constant is an absolute one.
This estimate is also valid for real $R,\lambda_1$ and $\lambda_2$,
however the implied constant may then depend on $q_1$ and $q_2$.
We set
\begin{gather*}
\lambda_1=\frac{4\log z}{9\log q_1},
\quad
\lambda_2=\frac{4\log z}{9\log q_2}
\quad
\text{ and }
\quad
R=z^{2/9}
.
\end{gather*}
Then clearly $R\in [1,z]$ and a short calculation shows that all four summands in the error term are $\ll z^{8/9}$,
which proves the first part.
For the second part we make use of Lemma~\ref{lem_fourier_coeff_estimate} and Parseval's identity to estimate~(\ref{eqn:independence_first_part}) by
\begin{multline}
  \frac{z^2}{R^2}
  \,
  \sup_{h\in\dZ}
  \abs{F_{q_1,\lambda_1}(h,\alpha)}^2
  \,
  \sup_{t\in \dR}
  \abs{
    \sum_{h_1<q_1^{\lambda_1}}{
      \min\left\{R^2,\Abs{h_1/q_1^{\lambda_1}+t}^{-2}\right\}
    }
  }
\\
  \times
  \sum_{h_2<q_2^{\lambda_2}}
  \abs{F_{q_2,\lambda_2}(h_2,\beta)}^2
\ll
  z^2
  q_1^{-2c\lambda_1}
  \frac{q_1^{\lambda_1}}{R}
,
\end{multline}
where $c=c_{q_1}\Abs{(q_1-1)\alpha}^2$.
Therefore for some constant $C$
the following holds for all $x,z\geq 0$ and all integers $R\in [1,z]$.
\begin{multline}\label{eqn:independence_all_parts}\nonumber
    \abs{\sum_{x<n\leq x+z}\e\p{\alpha s_{q_1}(n)+\beta s_{q_2}(n)}}
\\
  \leq
    C\p{
      q_1^{\lambda_1}q_2^{\lambda_2}
    +
      z^{1/2}R^{1/2}
    +
      zR^{1/2}\p{q_1^{-\lambda_1/2}+q_2^{-\lambda_2/2}}
    +
      zq_1^{\lambda_1(1/2-c)}R^{-1/2}
    }
.
\end{multline}
Again we may assume that $R,\lambda_1$ and $\lambda_2$ are real numbers.
We set
\begin{gather*}
\lambda_1=\frac{2\log z}{(4+c)\log q_1},
\quad
\lambda_2=\frac{2\log z}{(4+c)\log q_2}
\text{ and }
\quad
R=z^{\frac{2-2c}{4+c}}
.
\end{gather*}
With these choices we get after a short calculation
\begin{equation*}
    \sum_{x<n\leq x+z}\e\p{\alpha s_{q_1}(n)+\beta s_{q_2}(n)}
  \ll
    z^{1-c/(4+c)}
.
\end{equation*}
To get a convenient form of the exponent,
we note that $q_1\geq 2$, which implies $c_{q_1}\geq \pi^2/(36\log q_1)$.
By the same condition and monotonicity of $x/(4+x)$ we get
\[
  \frac c{4+c}
\geq
  \frac{\pi^2\Abs{(q_1-1)\alpha}^2}{36\log q_1\p{4+\frac{\pi^2\Abs{(q_1-1)\alpha}^2}{36\log q_1}} }
\geq
  \frac{\Abs{(q_1-1)\alpha}^2}{\frac{144\log q_1}{\pi^2}+\frac 14}
\geq
  \frac{\Abs{(q_1-1)\alpha}^2}{15\log q_1}.
\]
\end{proof}
By Corollary~\ref{cor_1} and~(\ref{eqn:independence_exponential_sum}) we see that for all real $\alpha$ and $\beta$
the function $\varphi(m)=\e\p{\alpha s_{q_1}(m)+\beta s_{q_2}(m)}$
admits a ``change of variables'' as long as $2-(8/9+1)c>0$, that is, $c<18/17$.
We assume now that
$(q_1-1)\alpha\not\in\dZ$
or
$(q_2-1)\beta\not\in\dZ$.
Then by partial summation and equation~(\ref{eqn:independence_sup_bound})
the second sum in~(\ref{eqn:cor_expsum_conclusion}) can be eliminated,
leading to the following statement:

Let $q_1,q_2\geq 2$ be relatively prime and $\alpha,\beta\in\dR$
such that $(q_1-1)\alpha\not\in\dZ$ or $(q_2-1)\beta\not\in\dZ$.
Then for all $c\in (1,18/17)$
there exist $\varepsilon>0$ and $C$ such that for $N\geq 1$ we have
\[
  \sum_{1\leq n\leq N}\e\p{\alpha s_{q_1}\p{\floor{n^c}}+\beta s_{q_2}\p{\floor{n^c} }}
\leq
  CN^{1-\varepsilon}
.
\]
From this exponential sum estimate we get the statement of Theorem~\ref{app_2}
by an orthogonality argument, which completes the proof.

Note that by the same orthogonality argument~(\ref{eqn:joint_introduction})
can be deduced from
from~(\ref{eqn:independence_sup_bound}),
which gives an alternative to Kim's proof~\cite{K99}.
\subsection{The Zeckendorf sum-of-digits function}\label{subsection_Zeckendorf}
In our third application we study the distribution in residue classes
of the values of the Zeckendorf sum-of-digits function on $\floor{n^c}$.

For $k\geq 0$ let $F_k$ be the $k$-th Fibonacci number, that is,
$F_0=0$, $F_1=1$ and $F_k=F_{k-1}+F_{k-2}$ for $k\geq 2$.
By Zeckendorf's Theorem~\cite{Z72} every positive integer $n$ admits a unique representation
\[
n=\sum_{i\geq 2}\varepsilon_iF_i
,
\]
where $\varepsilon_i\in \{0,1\}$ and $\varepsilon_i=1\Rightarrow \varepsilon_{i+1}=0$.
By this theorem we may write the $i$-th coefficient $\varepsilon_i$ as a function of $n$.
The Zeckendorf sum-of-digits of $n$ is then defined as
\[
  s_Z(n)
=
  \sum_{i\geq 2}\varepsilon_i(n)
.
\]
We set $s_Z(0)=0$.
We note that $s_Z(n)$ is the least $k$ such that $n$ is the sum of $k$ Fibonacci numbers.
\begin{thm}[The Zeckendorf sum-of-digits function on $\floor{n^c}$]\label{app_3}
Let $m\geq 1$ and $a$ be integers.
Then for all $c\in (1,4/3)$
there exists $\varepsilon>0$ such that
uniformly for $x\geq 1$ we have
\begin{equation}\label{eqn:zeckendorf_residue_classes}\nonumber
    \abs{\left\{n\leq x: s_Z\p{\floor{n^c}}\equiv a\bmod m\right\}}
  =
    \frac xm
  +
    O\p{x^{1-\varepsilon}}
.
\end{equation}
\end{thm}
The proof of this statement is based on the following proposition.
\begin{prp}\label{prp_zeckendorf_fourier_expression}
There exist $C$ such that for all $\alpha\in\dR$ and $z\geq 1$ we have
\begin{equation}\label{eqn:zeckendorf_integral}
\int_0^1\sup_{x\geq 0}\abs{\sum_{x<n\leq x+z}\e\p{\alpha s_Z(n)+n\theta}}\ud\theta\leq Cz^{1/2}.
\end{equation}
Moreover for $\alpha\not\in\dZ$ there exist $\eta>0$ and $C_1$ such that for all $z\geq 1$
\begin{equation}\label{eqn:zeckendorf_abs}
\sup_{x\geq 0}\abs{\sum_{x<n\leq x+z}\e\p{\alpha s_Z(n)}}\leq C_1z^{1-\eta}
.
\end{equation}
\end{prp}
\begin{proof}
For $k\geq 0$ we define
\[
G_k(\alpha,\theta)=\sum_{0\leq u<F_k}\e\p{\alpha s_Z(u)+\theta u}
.
\]
By the Cauchy-Schwarz inequality and the formula $F_k\asymp \varphi^k$,
where $\varphi=(\sqrt{5}+1)/2$,
we clearly have
\begin{equation}\label{eqn:zeckendorf_integral_block}
\int_0^1\abs{\sum_{n<F_k}G_k(\alpha,\theta)}\ud\theta\leq F_k^{1/2}\ll \varphi^{k/2}.
\end{equation}
Moreover, by the relation
$s_Z(u+F_k)=1+s_Z(u)$
that holds for $k\geq 2$ and $0\leq u<F_{k-1}$
the terms $G_k(\alpha,0)$ satisfy the linear recurrence relation
\[
G_{k+1}(\alpha,0)
=
G_k(\alpha,0)
+
\e(\alpha)G_{k-1}(\alpha,0)
.
\]
Its characteristic polynomial has the roots
$\tfrac 12\pm \tfrac 12\sqrt{1+4\e(\alpha)}$, whose absolute values are
bounded by
$\tfrac 12+\tfrac 12(17+8\cos(2\pi\alpha))^{1/4}$.
This expression is equal to $\varphi$ if $\alpha\in\dZ$ and
strictly less than $\varphi$ otherwise.
Consequently, if $\alpha\not\in\dZ$, there is some $\eta>0$ such that
\begin{equation}\label{eqn:zeckendorf_abs_block}
G_k(\alpha,0)\ll \varphi^{k(1-\eta)}
.
\end{equation}
The expression for $G_k(\alpha,\theta)$ involves a sum over the interval
$[0,F_k)$.
In order to deal with arbitrary finite intervals $I$ in $\dN$,
we decompose the interval $I$ according to the Zeckendorf representation of its endpoints.
This procedure is analogous to the decomposition of an interval into dyadic intervals,
which we used in the proof of Theorem~\ref{app_1}.
\begin{lem}\label{lem_phi-adic}
Let $0\leq A<B$ be integers.
There exist integers $L\geq 2$ and $a_j$, $b_j$ for $2\leq j\leq L$ such that
$A=a_2\leq \cdots\leq a_L=b_L\leq\cdots \leq b_2=B$
having the properties that
$\varepsilon_i(a_j)=\varepsilon_i(b_j)=0$ for $2\leq i<j\leq L$ and that
$a_{j+1}-a_j\in\{0,F_{j-1}\}$ and
$b_j-b_{j+1}\in\{0, F_j\}$ for $2\leq j<L$.
\end{lem}
\begin{proof}
We first show that it is sufficient to assume that
$0\leq A<F_K\leq B<F_{K+1}$ for some $K\geq 2$.
Let $K=\max\{i:\varepsilon_i(A)\neq\varepsilon_i(B)\}$ and
$C=\sum_{i>K}\varepsilon_i(A)F_i=\sum_{i>K}\varepsilon_i(B)F_i$.
Then $0\leq A-C<F_K\leq B-C<F_{K+1}$ and
by our assumption we get a decomposition
$A-C=a_2\leq\cdots\leq a_L=b_L\leq\cdots\leq b_2=B-C$ as in the Lemma.
We have $\varepsilon_i(a_j)=\varepsilon_i(b_j)=0$ for $2\leq j\leq L$ and $i>K$
and since $\varepsilon_K(B)=1$, we have $\varepsilon_i(C)=0$ for $i\leq K+1$.
Therefore $A=a_2+C\leq\cdots\leq a_L+C=b_L+C\leq\cdots\leq b_2+C=B$
is a valid decomposition of the interval $[A,B]$.

It remains to prove the simplified statement.
In the case that $A=0$ we set
$a_2=\ldots=a_{K+1}=0$ and
$b_j=\sum_{i\geq j}\varepsilon_i(B)F_i$ for $2\leq j\leq K+1$.
Otherwise we set
$b_j=\sum_{i\geq j}\varepsilon_i(B)F_i$ for $2\leq j\leq K$
and to choose $a_j$, we use the following assertion which we prove by (downward) induction on $k$.
\begin{itemize}
\item[]
Let $K\geq 2$.
Assume that $0<A\leq F_K$ and $k=\min\{i:\varepsilon_i(A)=1\}$.
There exist integers $A=a_k\leq\cdots\leq a_K=F_K$ such that
for $k\leq j<K$ and $2\leq i<j$ we have
$\varepsilon_i(a_j)=0$ and
$a_{j+1}-a_j\in\{0,F_{j-1}\}$.
\end{itemize}
If $k=K$, then $A=F_K$ and we choose $a_K=A$.
Otherwise $2\leq k<K$ and we set $A'=A+F_{k-1}$ and $k'=\min\{i:\varepsilon_i(A')=1\}$.
Then $k'>k$.
We choose $a_{k'},\ldots,a_K$ according to the assumption,
$a_k=A$ and $a_{k+1}=\cdots=a_{k'-1}=A'$.
This choice gives an admissible decomposition of the interval $[A,F_K]$ and the
statement is proved.
Setting $a_2=\cdots=a_{k-1}=A$ completes the proof of Lemma~\ref{lem_phi-adic}.
\end{proof}
By this lemma we can decompose an arbitrary finite interval in $\dN$
into intervals of the form
$[A,A+F_j)$,
where $\varepsilon_i(A)=0$ for $i\leq j$,
in such a way that for each $j\geq 1$ there are
at most $2$ intervals of this form.
Noting also that $s_Z(n)=s_Z(A)+s_Z(n-A)$ for all $n$ in such an interval
and using the formula $F_k\asymp \varphi^k$,
one can easily derive~(\ref{eqn:zeckendorf_integral}) and~(\ref{eqn:zeckendorf_abs})
from~(\ref{eqn:zeckendorf_integral_block}) and~(\ref{eqn:zeckendorf_abs_block}).
\end{proof}
We plug~(\ref{eqn:zeckendorf_integral}) into Corollary~\ref{cor_1}
and eliminate the second sum in~(\ref{eqn:cor_expsum_conclusion}) by partial summation and~(\ref{eqn:zeckendorf_abs}),
which results in the statement that
for $\alpha\in\dR\setminus\dZ$
and for $c\in (1,4/3)$
there exist $\eta>0$ and $C$
such that
\begin{equation*}
\sum_{1\leq n\leq N}\e\p{\alpha s_Z\p{\floor{n^c}} }\leq CN^{1-\eta}
\end{equation*}
for $N\geq 1$.
By transferring this to a statement about residue classes,
we obtain the statement of Theorem~\ref{app_3}.
\section{Proofs of the main results}
We start with a couple of lemmas that we need in the proofs of
Theorem~\ref{thm_1} and Proposition~\ref{prp_1}.
The first one will allow proving
that the left hand sides of~(\ref{eqn:expsum_conclusion}) and~(\ref{eqn:essential_conclusion}) are always $O(A)$.
\begin{lem}\label{lem_lhs_bounded}
Let $f:\dR^+\ra\dR^+$ be differentiable and assume that $f'$ is increasing and positive.
Then
\[\sum_{f(A)<m\leq f(2A)}\p{f^{-1}}'(m)\ll A\]
for $A>0$.
\end{lem}
\begin{proof}
If $g:\dR^+\ra\dR^+$ is decreasing and $0<s\leq t$, we have
\begin{multline*}
    \sum_{s<m\leq t}g(m)
  =
    \sum_{\floor{s}+1\leq m\leq\floor{t}}g(m)
  =
    \int_{\floor{s}}^{\floor{t}}g\p{\floor{x}+1}\ud x
\\
  \leq
      g\p{\floor{s}+1}
    +
      \int_{\floor{s}+1}^{\floor{t}}g(x)\ud x
  \leq
      g(s)
    +
      \int_s^tg(x)\ud x
.
\end{multline*}
We apply this to the function $g(x)=\p{f^{-1}}'(x)$,
noting also that there is some $a>0$
such that the sum in the lemma is equal to $0$ for $A<a$.
For $A\geq a$ we have
\[
    \sum_{f(A)<m\leq f(2A)}\p{f^{-1}}'(m)
  \leq
      \frac 1{f'(A)}
    +
      f^{-1}(x)\Big|_{f(A)}^{f(2A)}
  \leq
      \frac 1{f'(a)}
    +
      A
  \ll
      A
.
\]
\end{proof}
In the next lemma we study properties of functions $f$ as in Theorem~\ref{thm_1} and Proposition~\ref{prp_1}.
\begin{lem}\label{lem_properties_f}
Assume that $f:\dR^+\ra\dR$ is two times continuously differentiable,
$f,f',f''>0$
and that there exist $c_1\geq 1/2$ and $c_2>0$
such that for $0<x\leq y\leq 2x$ we have
$c_1f''(x)\leq f''(y)\leq c_2f''(x)$.
Then the following estimates hold.
\begin{alignat}{3}
  xf''(x) &\ll yf''(y)&&
  &&
    \text{ for }0<x\leq y
    \label{eqn:almost_monotone}
\\
    xf''(x) &\ll f'(x)&&\ll xf''(x)\log x
    \quad
  &&
    \text{ for }x\geq 2
    \label{eqn:df_d2f}
\\
    f'(x) &\leq f'(y)&&\ll f'(x)
  &&
    \text{ for }0<x\leq y\leq 2x,
    \label{eqn:df_quotient}
\\
    \log x &\ll f'(x)&&\ll x^\delta
  &&
    \text{ for some }\delta\geq 0\text{ and all }x\geq 2
    \label{eqn:df_x}
.
\end{alignat}
Moreover for $0<x\leq a\leq b\leq 2x$ we have
\begin{equation}\label{eqn:mvt_mon_1}
f(b)-f(a)\asymp f'(x)(b-a)
\end{equation}
and
\begin{equation}\label{eqn:mvt_mon_2}
f'(b)-f'(a)\asymp f''(x)(b-a).
\end{equation}
\end{lem}
\begin{proof}
In order to prove~(\ref{eqn:almost_monotone}),
we show the equivalent statement that
\[
f''(x)\ll af''(ax)
\]
for $a\geq 1$ and $x>0$.
This is clear for $a=2^k$ by the inequalities
$c_1f''(x)\leq f''(2x)$ and $c_1\geq 1/2$.
If $2^k\leq a<2^{k+1}$, we have
$f''(ax)\geq c_1f''(2^kx)\geq c_12^{-k}f''(x)\gg 1/af''(x)$.
We turn to the first inequality in~(\ref{eqn:df_d2f}).
By the Mean Value Theorem there exists some $\xi\in [x/2,x]$ such that
$f'(x)
\geq
f'(x)-f'(x/2)
=
x/2f''(\xi)
\geq
x/(2c_2) f''(x)
$.
For the proof of the second inequality in~(\ref{eqn:df_d2f}),
let $x\geq 2$. For $t\leq x$ we have
$tf''(t) \ll xf''(x)$
by~(\ref{eqn:almost_monotone})
and therefore
\[
    f'(x)=f'(2)+\int_2^xf''(t)\ud t
  \ll
    f'(2)+xf''(x)\int_2^x\frac 1t\ud t
  \leq
    f'(2)+xf''(x)\log x
.
\]
For $x\geq 2$ we have
$xf''(x)\log x\gg f''(2)\gg f'(2)$
by~(\ref{eqn:almost_monotone})
and $f',f''>0$,
therefore
$f'(x)\ll xf''(x)\log x$.
The first inequality of~(\ref{eqn:df_quotient}) is obvious since $f'$ is increasing.
By applying the Mean Value Theorem
it follows that there exists
$\xi\in [x,2x]$ such that
$f'(2x)-f'(x)=xf''(\xi)\ll xf''(x)$.
Together with~(\ref{eqn:df_d2f}) we get
$f'(2x)\ll f'(x)$.
We prove~(\ref{eqn:df_x}).
The first estimate follows from~(\ref{eqn:almost_monotone}) if we set $x=1$ and
integrate in $y$.
By~(\ref{eqn:df_quotient})
there exists $c>0$ such that $f'(2z)\leq cf'(z)$ for all $z>0$, from which we get
$f'(x)\ll c^\frac{\log x}{\log 2}f'(1)$ for all $x\geq 1$.
Let $0<x\leq a\leq b\leq 2x$.
By the Mean Value Theorem there is some $\xi\in[a,b]$ such that $f(b)-f(a)=f'(\xi)(b-a)$.
From the monotonicity of $f'$ and~(\ref{eqn:df_quotient}) we get~(\ref{eqn:mvt_mon_1}).
Analogously,~(\ref{eqn:mvt_mon_2}) is proved
via the assumption $c_2f''(x)\leq f''(y)\leq c_2f''(x)$.
\end{proof}
In the following lemma we integrate over a well-known estimate
for the exponential sum $\sum\e(nx)$, where the sum extends over an interval containing $B$ integers.
\begin{lem}\label{lem_integral_1}
Let $a\leq b$ be real numbers and $B\geq 2$.
Then
\begin{equation*}
    \int_{a}^{b}{
      \min\left\{B,\Abs{x}^{-1}\right\}
    }
    \ud x
  \leq
    2\p{b-a+1}\p{1+\log B}
.
\end{equation*}
\end{lem}
\begin{proof}
Since the integrand is $1$-periodic and symmetric with respect to $\frac 12$,
we have
\begin{multline*}
    \int_a^b{
      \min\left\{B,\Abs{x}^{-1}\right\}
    }
    \ud x
  \leq
    2(b-a+1)
    \int_0^{1/2}{
      \min\left\{B,\Abs{x}^{-1}\right\}
    }
    \ud x
\\
  \leq
      2(b-a+1)
      \p{
        \int_0^{1/B}{
          B
        }
        \ud x
      +
        \int_{1/B}^{1/2}{
          x^{-1}
        }
        \ud x
      }
\\
  \leq
    2(b-a+1)
    \p{1+\log (1/2)-\log (1/B)}
  \leq
    2\p{b-a+1}\p{1+\log B}
.
\end{multline*}
\end{proof}
\subsection{Proof of Proposition~\ref{prp_1}}\label{subsection_proof_prp_1}
We prepare for the proof of Proposition~\ref{prp_1} by giving some results on
the approximation of a twice differentiable function by an affine linear function.
\begin{lem}\label{lem_taylor}
Let $f:[a,b]\ra\dR$ be twice differentiable and $\abs{f''}\leq M$.
For all $\alpha\in f'([a,b])$ and $a\leq x\leq b$ we have
\[
    \abs{x\alpha+f(a)-a \alpha-f(x)}
  \leq
    M(b-a)^2
.
\]
\end{lem}
\begin{proof}
By the Mean Value Theorem there exists some $\xi_1\in [a,x]$ such that
$f(x)-f(a)=f'(\xi_1)(x-a)$, that is, such that
$\abs{x\alpha+f(a)-a\alpha-f(x)}=
(x-a)\abs{f'(\xi_1)-\alpha}$.
There exists some $y\in[a,b]$ such that $\alpha=f'(y)$.
By applying the Mean Value Theorem to the function $f'$, we get some
$\xi_2$ between $\xi_1$ and $y$ such that
$\abs{f'(\xi_1)-\alpha}=\abs{f'(\xi_1)-f'(y)}=\abs{(\xi_1-y)f''(\xi_2)}$.
From this the statement follows easily.
\end{proof}
The following result will permit us to replace the function
$\floor{f(n)}$ by a Beatty sequence on an interval
$(a,b]$.
\begin{lem}\label{lem_beattyfication}
Let $f:[a,b]\ra\dR$ be twice differentiable and $\abs{f''}\leq M$.
For all $\alpha\in f'([a,b])$ and $a\leq x\leq b$ such that
$\Abs{x\alpha+f(a)-a\alpha}>M(b-a)^2$ we have
\[\floor{f(x)}=\floor{x\alpha+f(a)-a\alpha}.\]
\end{lem}
\begin{proof}
We write $\beta=f(a)-a\alpha$ and $d=M(b-a)^2$.
The condition $\Abs{x\alpha+\beta}>d$ in the statement of the lemma
implies
$\floor{x\alpha+\beta-d}=\floor{x\alpha+\beta}=\floor{x\alpha+\beta+d}$.
Moreover by Lemma~\ref{lem_taylor} we get
$x\alpha+\beta-d\leq f(x)\leq x\alpha+\beta+d$.
Combining these observations yields the claim.
\end{proof}
We estimate the number of integers in an interval for which such an approximation fails.
\begin{lem}\label{lem_approx_lemma}
Let $a\leq b$ be integers and let $f:[a,b]\ra\dR$ be twice differentiable.
Assume that $\abs{f''}\leq M$.
For all $\alpha\in f'([a,b])$ and all $R\geq 1$ we have the estimate
\begin{multline}\label{eqn:approx_error}\nonumber
    \abs{\{n\in (a,b]:\floor{f(n)}\neq\floor{n\alpha+f(a)-a\alpha}\}}
\\
  \leq
    2M(b-a)^3+
    \frac {(b-a)}R+
    \sum_{1\leq r\leq R}
    \frac 1r\abs{\sum_{a<n\leq b}\e\p{nr\alpha}}
.
\end{multline}
\end{lem}
\begin{proof}
Write $d=M(b-a)^2$ and $\beta=f(a)-a\alpha$.
If $d\geq \frac 12$ or $a=b$
the statement follows immediately since the left hand side is bounded by $b-a$.
Otherwise it suffices by Lemma~\ref{lem_beattyfication} to estimate the quantity
\[
\abs{\{n\in (a,b]:\Abs{n\alpha+\beta}\leq d\}}
.
\]
To do this, we apply the inequality of
Erd\H{o}s and Tur\'{a}n
to the sequence
$\p{\{n\alpha+\beta+d\}}_{a<n\leq b}$ in
$[0,1)$.
According to~\cite[Lemma 1]{MRS02}, the discrepancy of any real valued finite sequence $(x_1,\ldots,x_N)$ in
$[0,1)$,
where $N\geq 1$, satisfies
\begin{multline*}
    D_N(x_1,\ldots,x_N)
  =
    \sup_{0\leq r\leq s<1}\abs{\frac 1N\abs{\{1\leq n\leq N:r\leq x_n\leq s\}}-(s-r)}
\\
  \leq
    \frac 1{H+1}+\sum_{1\leq h\leq H}\frac 1h\abs{\frac 1N\sum_{1\leq n\leq N}\e(hx_n)}
\end{multline*}
for all $H\geq 1$.
This is the classical inequality of Erd\H{o}s and Tur\'{a}n with an improved constant, equal to $1$.

Considering the interval $[0,2d]$, we obtain from this the estimate
\begin{multline*}
    \abs{\frac 1{b-a}\abs{\{n\in (a,b]:\Abs{n\alpha+\beta}\leq d\}}-2d}
\\
  =
    \abs{\frac 1{b-a}\abs{\{n\in (a,b]:\{n\alpha+\beta+d\}\in [0,2d]\}}-2d}
\\
  \leq
      \frac 1R
    +
      \frac 1{b-a}
      \sum_{1\leq r\leq R}{
        \frac 1r\abs{\sum_{a<n\leq b}\e\p{nr\alpha+r\beta+rd}}
      }
,
\end{multline*}
from which the claim follows.
\end{proof}
The rough idea of the proof of Proposition~\ref{prp_1} is
to relate the two sums in~(\ref{eqn:essential_conclusion}) to each other in three steps,
introducing the expression~(\ref{eqn:essential_integral}).
We replace the function $\floor{f(n)}$
by a Beatty sequence $\floor{n\alpha+\beta}$ on small subintervals of
$(A,2A]$.
Analogously, we replace the expression $\p{f^{-1}}'(m)$ by the constant value $\frac 1\alpha$
on corresponding subintervals of
$(f(A),f(2A)]$.
To link the two expressions thus obtained we insert~(\ref{eqn:essential_integral}),
which expresses the error that arises when we replace the sum of $\varphi(n)$ over a Beatty sequence
by a sum of $\varphi(n)$ over all integers in an interval.
Afterwards we collect the error terms and we are done.
\begin{proof}[Proof of Proposition~\ref{prp_1}]
Let $A\geq 2$.
It is sufficient to concentrate on the case that $K$ is an integer and $2\leq K\leq A$, for the following reasons.
If $K<2$, then $\frac{(\log A)^2}K\gg 1$, and if $K>A$, then
$f''(A)K^2\geq Af''(A)A\gg 2f''(2)\gg 1$ by~(\ref{eqn:almost_monotone}).
Therefore the right hand side of~(\ref{eqn:essential_conclusion}) is bounded below for these cases, while
the left hand side of~(\ref{eqn:essential_conclusion}) is always bounded above by Lemma~\ref{lem_lhs_bounded}.
For general $K$ in $[2,A]$ we have $\abs{I(A,\floor{K})-I(A,K)}\ll \frac 1K$,
which can be deduced from the inequality
$\abs{ab-a'b'}\leq \abs{a-a'}\abs{b}+\abs{a'}\abs{b-b'}$
and the estimate $\alpha\geq f'(2)\gg 1$ that is valid for $\alpha\in[f'(A),f'(2A)]$.
This error is absorbed by the term $\frac{(\log A)^2}K$,
therefore the general case can easily be accounted for by adjusting the implied constant $C$.

To guarantee that all expressions involving $\varphi$ are well-defined, we set $\varphi(n)=0$ for $n\leq 0$.
For $K$ an integer and $2\leq K\leq A$
we partition the interval
$(A,2A]$
into smaller intervals of length at most $K$ as follows.
Define integral partition points $a_i=\ceil{A}+iK$ for $i\geq 0$
and set $L=\max\{i:a_i\leq 2A\}$, which is well defined since $K>0$.
The integer $L$ satisfies the estimate $L\leq\frac AK$.
We have the decomposition
\begin{equation}\label{eqn:interval_decomposition}
  (A,2A]
=
  \left(A,\ceil{A}\right]
  \cup
  \bigcup_{0\leq i<L}(a_i,a_{i+1}]
  \cup
  (a_L,2A]
.
\end{equation}
Let $\alpha\in\dR$.
Then by the triangle inequality and the relation $a_{i+1}-a_i=K$ we have for $i<L$
\begin{multline}\label{eqn:elementary_split}
    \abs{
        \sum_{a_i<n\leq a_{i+1}}\varphi\p{\floor{f(n)}}
      -
        \sum_{f(a_i)<m\leq f(a_{i+1})}\varphi(m)\p{f^{-1}}'(m)
    }
\\
  \leq
      T_1(\alpha,i)
    +
      T_2(\alpha,i)
    +
      T_3(\alpha,i)
    +
      T_4(\alpha,i)
,
\end{multline}
where
\begin{flalign*}
    T_1(\alpha,i)
  &=
    \abs{
        \sum_{a_i<n\leq a_{i+1}}\bigl(
            \varphi\p{\floor{f(n)}}
          -
            \varphi\p{\floor{n\alpha+f(a_i)-a_i\alpha}}
        \bigr)
    }
,
\\
    T_2(\alpha,i)
  &=
    \abs{
        \sum_{0<n\leq K}\varphi\p{\floor{n\alpha+f(a_i)}}
      -
        \frac 1\alpha\sum_{f(a_i)<m\leq f(a_i)+K\alpha}\varphi(m)
    }
,
\\
    T_3(\alpha,i)
  &=
    \abs{
        \frac 1\alpha\sum_{f(a_i)<m\leq a_{i+1}\alpha+f(a_i)-a_i\alpha}\varphi(m)
      -
        \frac 1\alpha\sum_{f(a_i)<m\leq f(a_{i+1})}\varphi(m)
    }
,
\\
    T_4(\alpha,i)
  &=
    \abs{
      \sum_{f(a_i)<m\leq f(a_{i+1})}{
        \varphi(m)\p{
            \frac 1\alpha
          -
            \p{f^{-1}}'(m)
        }
      }
    }
.
\end{flalign*}
We integrate
~(\ref{eqn:elementary_split}) in $\alpha$ from $f'(a_i)$ to $f'(a_{i+1})$,
divide by the length of the integration range,
and take the sum over $i$ from $0$ to $L-1$, obtaining
\begin{multline}\label{eqn:split}
    \abs{
        \sum_{\ceil{A}<n\leq a_L}\varphi\p{\floor{f(n)}}
      -
        \sum_{f(\ceil{A})<m\leq f(a_L)}\varphi(m)\p{f^{-1}}'(m)
    }
\\
  \leq
    \sum_{0\leq i<L}
      \frac 1{f'(a_{i+1})-f'(a_i)}
      \int_{f'(a_i)}^{f'(a_{i+1})}
        \left(
          \vphantom{\sum}
          T_1(\alpha,i)
          +T_2(\alpha,i)
        \right.
\\
        \left.
          \vphantom{\sum}
          +T_3(\alpha,i)
          +T_4(\alpha,i)
        \right)
      \ud \alpha
.
\end{multline}
The first summand will be estimated with the help of Lemma~\ref{lem_approx_lemma},
the second by $A\,I(A,K)$,
and the third and fourth terms will be estimated trivially.

We estimate the first summand in~(\ref{eqn:split}).
If $R$ is a positive integer, $0\leq i<L$ and
$\alpha\in f'([a_i,a_{i+1}])$, Lemma~\ref{lem_approx_lemma} gives
\begin{equation}\label{eqn:beatty_erdos_turan}
    T_1(\alpha,i)
  \leq
      2f''(A)K^3
    +
      \frac KR
    +
      \sum_{1\leq r\leq R}{
        \frac 1r
        \abs{\sum_{a_i<n\leq a_{i+1}}\e(nr\alpha)}
      }
.
\end{equation}
By~(\ref{eqn:mvt_mon_2}) we have
$f'(2A)-f'(A)\ll Af''(A)$ and
$f'(a_{i+1})-f'(a_i)\gg f''(A)K$ for $0\leq i<L$.
Note also that $Af''(A)\gg 2f''(2)>0$ for all $A\geq 2$
by~(\ref{eqn:almost_monotone})
and $f''>0$.
From Lemma~\ref{lem_integral_1} it follows that for $2\leq K\leq A$ and $r\geq 1$ we have
\begin{multline}\label{eqn:kuzmin_landau_integral}
    \sum_{0\leq i<L}{
      \frac 1{f'(a_{i+1})-f'(a_i)}
      \int_{f'(a_i)}^{f'(a_{i+1})}{
        \abs{\sum_{a_i<n\leq a_{i+1}}\e(nr\alpha)}
      }
      \ud \alpha
    }
\\
  \ll
    \frac 1{f''(A)K}
    \sum_{0\leq i<L}{
      \frac 1r
      \int_{rf'(a_i)}^{rf'(a_{i+1})}{
        \abs{\sum_{a_i<n\leq a_{i+1}}\e(xn)}
      }
      \ud x
    }
\\
  \leq
    \frac 1{f''(A)K}
    \frac 1r
    \int_{rf'(A)}^{rf'(2A)}{
      \min\left\{K,\Abs{x}^{-1}\right\}
    }
    \ud x
\\
  \ll
    \frac 1{f''(A)K}
    \frac 1r
    2(rAf''(A)+1)\p{1+\log K}
  \ll
    A\frac{\log K}{K}
.
\end{multline}
From~(\ref{eqn:beatty_erdos_turan}) and~(\ref{eqn:kuzmin_landau_integral})
and the estimates $L\leq\frac AK$ and $\sum_{i=1}^R\frac 1r\leq \log R+1$
it follows that for $2\leq K\leq A$ and $R\geq 2$ we have
\begin{multline}\label{eqn:first_step}
    \sum_{0\leq i<L}{
      \frac 1{f'(a_{i+1})-f'(a_i)}
      \int_{f'(a_i)}^{f'(a_{i+1})}{
        T_1(\alpha,i)
      }
      \ud \alpha
    }
\\
  \ll
    \frac AK
      \p{
        f''(A)K^3+\frac KR
      }
    +
      \frac {A\log K(\log R+1)}K
\\
  \ll
    A\p{
      f''(A)K^2+\frac 1R+\frac{\log K\log R}K
    }
,
\end{multline}
which concludes our treatment of the first term in~(\ref{eqn:split}).
We turn to the second summand.
Again we use~(\ref{eqn:mvt_mon_2}) and obtain the estimates
\[
\frac 1{f'(a_{i+1})-f'(a_i)}\ll \frac 1{f''(A)K}=\frac AK\frac 1{Af''(A)}\ll A\frac 1{f'(2A)-f'(A)}\frac 1K
\]
for $0\leq i<L$.
By inserting this and the definition of $T_2(\alpha,i)$, we easily obtain
\begin{equation}\label{eqn:second_step}
    \sum_{0\leq i<L}{
      \frac{1}{f'(a_{i+1})-f'(a_i)}
      \int_{f'(a_i)}^{f'(a_{i+1})}{
        T_2(\alpha,i)
      }
      \ud\alpha
    }
  \ll
    A\,I(A,K)
.
\end{equation}
To estimate the third term in~(\ref{eqn:split}), assume that $0\leq i<L$
and $\alpha\in\bigl[f'(a_i),f'(a_{i+1})\bigr]$.
We use Lemma~\ref{lem_taylor} (setting $x=a_{i+1}$) to get
\[
\abs{a_{i+1}\alpha +f(a_i)-a_i\alpha-f(a_{i+1})}
\leq c_2f''(A)K^2,
\]
therefore the two sums in the definition of $T_3(\alpha,i)$ differ by not more than
$c_2f''(A)K^2+1$ summands. Moreover, we have $L\leq \frac AK$.
Estimating $\frac 1\alpha\leq \frac 1{f'(A)}$ we get
\begin{multline}\label{eqn:third_step}
    \sum_{0\leq i<L}{
      \frac 1{f'(a_{i+1})-f'(a_i)}
      \int_{f'(a_i)}^{f'(a_{i+1})}{
        T_3(\alpha,i)
      }
      \ud \alpha
    }
\\
  \ll
    \frac AK\frac 1{f'(A)}(f''(A)K^2+1)
  =
    A\p{
        \frac{f''(A)K}{f'(A)}
      +
        \frac 1{f'(A)K}
      }
.
\end{multline}
Finally let $0\leq i<L$, $\alpha\in f'([a_i,a_{i+1}])$ and $f(a_i)<m\leq f(a_{i+1})$.
Choose $x,y\in[a_i,a_{i+1}]$ in such a way that $\alpha=f'(x)$ and $m=f(y)$.
Then by~(\ref{eqn:mvt_mon_2}) and the monotonicity of $f'$ we have
\begin{multline*}
    \abs{\frac 1\alpha-\p{f^{-1}}'(m)}
  =
    \abs{\frac 1{f'(x)}-\frac 1{f'(y)}}
  =
    \abs{\frac {f'(y)-f'(x)}{f'(x)f'(y)}}
\\
  \leq
    \frac{f'(a_{i+1})-f'(a_i)}{f'(a_i)^2}
  \ll
    \frac{f''(A)K}{f'(A)^2}
.
\end{multline*}
Moreover, the length of summation in the definition of $T_4(\alpha,i)$ can be estimated
using~(\ref{eqn:mvt_mon_1}), giving $f(a_{i+1})-f(a_i)+1\ll f'(A)K+1$.
It follows that
\begin{multline}\label{eqn:fourth_step}
    \sum_{0\leq i<L}{
      \frac 1{f'(a_{i+1})-f'(a_i)}
      \int_{f'(a_i)}^{f'(a_{i+1})}{
        T_4(\alpha,i)
      }
      \ud x
    }
\\
  \ll
    \frac AK
    \p{
        f'(A)K+1
    }\p{
      \frac{f''(A)K}{f'(A)^2}
    }
  \ll
    A\p{
        \frac{f''(A)K}{f'(A)}
      +
        \frac{f''(A)}{f'(A)^2}
    }
.
\end{multline}
We still have to take care of the first and the last interval in~(\ref{eqn:interval_decomposition}).
To do this, we take any interval
$(a,b]$
such that $A\leq a\leq b\leq a+K\leq 2A$.
For all
$m\in \bigl(f(a),f(b)\bigr]$
we have
$\p{f^{-1}}'(m)=\frac 1{f'\p{f^{-1}(m)}}\leq \frac 1{f'(A)}$ since $f'$ is monotonic,
moreover $f(b)-f(a)+1\ll f'(A)K+1\ll f'(A)K$ by~(\ref{eqn:mvt_mon_1}) and the relation $f'(A)\geq f'(2)>0$,
and finally $b-a+1\ll K$. Therefore
\begin{multline}\label{eqn:endpieces}
    \abs{
        \sum_{a<n\leq b}{
          \varphi\p{\floor{f(n)}}
        }
      -
        \sum_{f(a)<m\leq f(b)}{
          \varphi(m)\p{f^{-1}}'(m)
        }
    }
\\
  \ll
    K+f'(A)K\frac 1{f'(A)}
  \ll K
.
\end{multline}
Combining
~(\ref{eqn:split}),
~(\ref{eqn:first_step}),
~(\ref{eqn:second_step}),
~(\ref{eqn:third_step}),
~(\ref{eqn:fourth_step}) and
~(\ref{eqn:endpieces}) we get
\begin{multline}\label{eqn:collection}\nonumber
    \abs{
        \sum_{A<n\leq 2A}{
          \varphi\p{\floor{f(n)}}
        }
      -
        \sum_{f(A)<m\leq f(2A)}{
          \varphi(m)\p{f^{-1}}'(m)
        }
    }
\\
  \ll
    A\left(
      f''(A)K^2+\frac 1R+\frac{\log K\log R}K+I(A,K)
    \right.
\\
    \left.
      +\frac{f''(A)K}{f'(A)}+\frac 1{f'(A)K}+\frac{f''(A)}{f'(A)^2}+\frac KA
    \right)
\end{multline}
for $A,K,R\geq 2$.
Since
$f'(A)\geq f'(2)\gg 1$,
the first term dominates the fifth and seventh terms
and the third term dominates the sixth.
Since $Af''(A)\gg 2f''(2)\gg 1$
by~(\ref{eqn:almost_monotone}), we have
$f''(A)\gg \frac 1A$,
and therefore the first term also dominates the last term.
We choose $R=A$. Then the third term dominates the second,
and the error is
\[
  \ll
    A\p{
      f''(A)K^2+\frac{(\log A)^2}K+I(A,K)
    }
.
\]
\end{proof}
\subsection{Proof of Theorem~\ref{thm_1}}
We want to find an estimate for~(\ref{eqn:essential_integral});
more precisely, we want to treat the expression
\[\sum_{a<n\leq b}\varphi\p{\floor{n\alpha+\beta}}\]
with the help of exponential sums.
To do this, we resort to a useful approximation of the
sawtooth function $x\mapsto \{x\}-\frac 12$ by trigonometric polynomials that was given by Vaaler.
(See~\cite[Theorem A.6]{GK91}.)
\begin{lem}\label{lem_vaaler}
Assume that $H$ is a
positive
integer.
There exist real numbers
$a_H(h)\in [0,1]$ for $1\leq\abs{h}\leq H$
such that
\begin{equation}\label{eqn:vaaler}
    \abs{\psi(t)-\psi_H(t)}
  \leq
    \kappa_H(t)
\end{equation}
for all real $t$, where
\[
\begin{aligned}
    \psi(x)&=\{x\}-\frac 12,\\
    \psi_H(t)&=-\frac{1}{2\pi i}\sum_{1\leq \abs{h}\leq H}\frac{a_H(h)}{h}\e(ht)\\
\end{aligned}
\]
and
\[
    \kappa_H(t)=\frac{1}{2(H+1)}\sum_{0\leq\abs{h}\leq H}\p{1-\frac{\abs{h}}{H+1}}\e(ht).
\]
\end{lem}
Note that $\kappa_H(t)$ is a nonnegative real number since for all $H$ we have
\[
    \sum_{0\leq \abs{h}<H}(H-\abs{h})e(hx)
  =
    \abs{\sum_{0\leq h<H}\e\p{hx}}^2
.
\]
Let $\alpha$ and $\beta$ be real numbers and suppose that $\alpha\geq 1$.
An elementary argument shows that for all integers $m$ we have
\begin{multline}\label{eqn:detection}
      \floor{-\frac{m-\beta}\alpha}
    -
      \floor{-\frac{m+1-\beta}\alpha}
\\
  =
    \left\{
      \begin{array}{ll}
            1
          &
            \text{if }m=\floor{n\alpha+\beta}
            \text{for some integer }n
        \\
            0
          &
            \text{otherwise}.
        \end{array}
    \right\}
.
\end{multline}
With the help of this characterization of the elements of a Beatty sequence we prove the following statement,
which allows us to deduce Theorem~\ref{thm_1} from Proposition~\ref{prp_1}.
\begin{prp}\label{prp_reduction}
Let $\varphi:\dN\ra\dC$ be a function bounded by $1$.
For all real $\alpha\geq 1$, $\beta\geq 0$, $K\geq 0$ and $H\geq 1$ we have
\begin{multline*}
    \abs{
        \sum_{0<n\leq K}\varphi\p{\floor{n\alpha+\beta}}
      -
        \frac 1\alpha\sum_{\beta<m\leq\beta+K\alpha}\varphi(m)
    }
\\
  \leq
      \sum_{1\leq \abs{h}\leq H}
      \min\left\{
          \frac 1\alpha
        ,
          \frac 1{\abs{h}}
      \right\}
      \abs{\sum_{\beta<m\leq\beta+K\alpha}\varphi(m)\e\p{-m\frac h\alpha}}
\\
    +
      \frac{1}{H}\sum_{0\leq \abs{h}\leq H}
      \abs{\sum_{\beta<m\leq\beta+K\alpha}\e\p{-m\frac h\alpha}}
    +
      O(1)
.
\end{multline*}
The implied constant is an absolute one.
\end{prp}
\begin{proof}
We write $\psi(x)=\{x\}-\frac 12=x-\floor{x}-\frac 12$.
Since $\alpha\geq 1$, the function $n\mapsto\floor{n\alpha+\beta}$ is injective.
Using this fact and~(\ref{eqn:detection}), we see that
\begin{multline*}
    \sum_{0<n\leq K}
    \varphi\p{\floor{n\alpha+\beta}}
\\
  =
    \sum_{m\in\dZ}
    \varphi(m)\cdot
    \left\{
      \begin{array}{ll}
        1&m=\floor{n\alpha+\beta}\text{ for some }0<n\leq K\\
        0&\text{otherwise}
      \end{array}
    \right\}
\\
  =
    \sum_{\floor{\beta}<m\leq\floor{\beta+K\alpha}}
    \varphi(m)\cdot
    \left\{
      \begin{array}{ll}
        1&m=\floor{n\alpha+\beta}\text{ for some }n\\
        0&\text{otherwise}
      \end{array}
    \right\}
\\
  =
    \sum_{\floor{\beta}<m\leq\floor{\beta+K\alpha}}
      \varphi(m)
      \p{
          \floor{-\frac{m-\beta}{\alpha}}
        -
          \floor{-\frac{m+1-\beta}{\alpha}}
      }
\\
  =
      \frac 1\alpha \sum_{\beta<m\leq\beta+K\alpha}\varphi(m)
\\
    +
      \sum_{\beta<m\leq\beta+K\alpha}
      \varphi(m)
      \p{
          \psi\p{-\frac{m+1-\beta}{\alpha}}
        -
          \psi\p{-\frac{m-\beta}{\alpha}}
      }
    +
      O(1)
.
\end{multline*}
It remains to treat the second sum.
For brevity, write
\[
L=\{m\in\dZ:\beta<m\leq\beta+K\alpha\}
.
\]
Let $H\geq 1$ be an integer.
For each $m$ we replace $\psi$ by $\psi_H$ with the help of~(\ref{eqn:vaaler}) to get
\begin{multline*}
  \left|
    \sum_{m\in L}{
      \varphi(m)
      \p{
          \psi\p{
            -\frac{m+1-\beta}{\alpha}-\gamma
          }
        -
          \psi\p{
            -\frac{m-\beta}{\alpha}-\gamma
          }
      }
    }
  \right.
\\
  -
  \left.
    \frac {-1}{2\pi i}
      \sum_{m\in L}{
        \varphi(m)
        \sum_{1\leq\abs{h}\leq H}{
          \frac{a_H(h)}{h}
          \p{
              \e\p{-h\frac{m+1-\beta}{\alpha}}
            -
              \e\p{-h\frac{m-\beta}{\alpha}}
          }
        }
      }
  \right|
\\
  \leq
    \frac{1}{2H+2}
    \sum_{m\in L}
    \sum_{\abs{h}\leq H}
      \p{1-\frac{\abs{h}}{H+1}}
      \p{
          \e\p{-h\frac{m+1-\beta}{\alpha}}
        +
          \e\p{-h\frac{m-\beta}{\alpha}}
      }
\\
  \leq
    \frac{1}{H+1}
    \sum_{0\leq \abs{h}\leq H}{
      \abs{
        \sum_{m\in L}{
          \e\p{-h\frac{m}{\alpha}}
        }
      }
    }
.
\end{multline*}
Finally we use the inequalities $\abs{a_H(h)}\leq 1$ and $\abs{\e(x)-1}\leq\min\{2, 2\pi x\}$ to calculate:
\begin{multline*}
    \abs{
      \frac 1{2\pi i}
      \sum_{m\in L}{
        \varphi(m)
        \sum_{1\leq\abs{h}\leq H}{
          \frac{a_H(h)}{h}
          \p{
              \e\p{-h\frac{m+1-\beta}{\alpha}}
            -
              \e\p{-h\frac{m-\beta}{\alpha}}
          }
        }
      }
    }
\\
  =
    \abs{
      \frac 1{2\pi}
      \sum_{1\leq\abs{h}\leq H}{
        \frac{a_H(h)}{h}
        \e\p{-\frac{\beta}\alpha}
        \p{\e\p{-\frac{h}\alpha}-1}
        \sum_{m\in L}{
          \varphi(m)
          \p{-h\frac m\alpha}
        }
      }
    }
\\
  \leq
    \sum_{1\leq\abs{h}\leq H}{
      \min\left\{
          \frac 1\alpha
        ,
          \frac 1{\abs{h}}
      \right\}
      \abs{
        \sum_{m\in L}{
          \varphi(m)
          \p{-h\frac m\alpha}
        }
      }
    }
.
\end{multline*}
If $H\geq 1$ is a real number, we apply these calculations to $\floor{H}$.
Note that in this process the summations over $h$ remain unchanged and that $1/(\floor{H}+1)\leq 1/H$,
therefore the assertion follows.
\end{proof}
We will use the following standard lemma
to extend the range of a summation in exchange for a controllable factor.
\begin{lem}\label{lem_vinogradov}
Let $x\leq y\leq z$ be real numbers and $a_n\in\dC$ for $x<n\leq z$.
Then
\begin{equation*}
    \abs{\sum_{x<n\leq y}a_n}
  \leq
    \int_0^1{
      \min\left\{y-x+1,\Abs{\xi}^{-1}\right\}
      \abs{\sum_{x<n\leq z}a_n\e\p{n\xi}}
    }
    \ud \xi
.
\end{equation*}
\end{lem}
\begin{proof}
Since
$\int_0^1\e\p{k\xi}\ud \xi=\delta_{k,0}$ for $k\in\dZ$
it follows that
\begin{equation*}
    \sum_{x<n\leq y}a_n
  =
    \sum_{x<n\leq z}a_n\sum_{x<m\leq y}\delta_{n-m,0}
  =
    \int_0^1
    \sum_{x<m\leq y}\e\p{-m\xi}
    \sum_{x<n\leq z}a_n\e\p{n\xi}
    \ud \xi,
\end{equation*}
from which the statement follows.
\end{proof}
Finally, to obtain the correct error term in the theorem,
we will use the following lower bound on the $L^1$-norm of an exponential sum.
\begin{lem}\label{lem_lower_bound}
Let $a<b$ be real numbers and $x_m$ a complex number for $a<m\leq b$.
Then
\[
  \int_0^1
  \abs{
    \sum_{a<m\leq b}
      x_m\e\p{m\theta}
  }
  \ud \theta
  \geq
  \max_{a<m\leq b}
  \abs{x_m}
.
\]
\end{lem}
\begin{proof}
For $a<n\leq b$ we have
\begin{multline*}
  \int_0^1
  \abs{
    \sum_{a<m\leq b}
      x_m\e\p{m\theta}
  }
  \ud \theta
  =
  \int_0^1
  \abs{
    \sum_{a<m\leq b}
      x_m\e\p{(m-n)\theta}
  }
  \ud \theta
\\
  \geq
  \abs{
    \sum_{a<m\leq b}
    x_m
    \int_0^1
      \e\p{(m-n)\theta}
    \ud \theta
  }
=x_n.
\end{multline*}
\end{proof}
\begin{proof}[Proof of Theorem~\ref{thm_1}]
Note first that by~(\ref{eqn:df_x}) we have $f'(x)\ra\infty$,
therefore there exists $A_0\geq 2$ such that
$f'(A)\geq 1$ for $A\geq A_0$.
Let $z>0$.
By an argument similar to that at the beginning of the proof of Proposition~\ref{prp_1}
we may restrict ourselves to the case that $z\leq A\,f'(A)$.
Also, we may assume that there exists an $m$ in the range $f(A)<m\leq f(2A)+z$ such that $\abs{\varphi(m)}=1$,
since the general case follows from this one by rescaling both sides of~(\ref{eqn:expsum_conclusion}).
To see this, we note that $A\geq 2$ is an integer and $f'(x)\geq 1$ for all $x\geq A$ and therefore
the relation~(\ref{eqn:expsum_conclusion}) only depends on integers $m$ in the range $f(A)<m\leq f(2A)+z$.
By Lemma~\ref{lem_lower_bound}, this restriction implies
\begin{multline}\label{eqn:J_lower_bound}
    \int_0^1
      \sup_{f(A)<x\leq f(2A)}
      \abs{
        \sum_{x<m\leq x+z}
        \varphi(m)\e\p{m\theta}
      }
    \ud \theta
\\
  \geq
    \sup_{f(A)<x\leq f(2A)}
    \int_0^1
      \abs{
        \sum_{x<m\leq x+z}
        \varphi(m)\e\p{m\theta}
      }
    \ud \theta
  \geq
    \sup_{f(A)<x\leq f(2A)}
    \sup_{x<m\leq x+z}
    \abs{\varphi(m)}
\\
  =
    \sup_{f(A)<m\leq f(2A)+z}
    \abs{\varphi(m)}
  \geq
    1
.
\end{multline}
If $z<\max\{2,f'(2A)\}$, this lower bound implies
$f'(A)(\log A)^3J(A,z)\gg 1$ and since by Lemma~\ref{lem_lhs_bounded} the left hand
side of~(\ref{eqn:expsum_conclusion}) is bounded, this proves the assertion in this case.
For the remaining part of the proof we assume therefore that
$\max\{2,f'(2A)\}\leq z\leq A\,f'(A)$.
Moreover, we assume throughout that
$1\leq K\leq A$ and that $H\geq 2$.
We want to apply Proposition~\ref{prp_1} and therefore we have to find an estimate for $I(A,K)$.
We apply Proposition~\ref{prp_reduction} to the expression in the absolute value in equation~(\ref{eqn:essential_integral}),
which is possible since $\alpha\geq f'(A)\geq 1$ for all $\alpha$ in question,
and obtain the estimate
\begin{multline}\label{eqn:reduction_applied}
      I(A,K)
    \ll
      \frac 1{f'(2A)-f'(A)}
      \frac 1K
      \int_{f'(A)}^{f'(2A)}
      \left(
          \sum_{1\leq \abs{h}\leq H}{
            \min\left\{
                \frac 1\alpha
              ,
                \frac 1{\abs{h}}
            \right\}
            S_1(\alpha,h)
          }
      \right.
\\
      \left.
        \vphantom{
          \sum_{1\leq \abs{h}\leq H}{
            \min\left\{
                \frac 1\alpha
              ,
                \frac 1{\abs{h}}
            \right\}
          }
        }
        +
          \frac 1H
          S_2(\alpha,0)
        +
          \frac 1H
          \sum_{1\leq \abs{h}\leq H}{
            S_2(\alpha,h)
          }
        +
          O(1)
      \right)
      \ud\alpha
,
\end{multline}
where
\begin{equation*}
      S_1(\alpha,h)
    =
      \sup_{
            f(A)<x\leq f(2A)
      }
      \abs{
        \sum_{x<m\leq x+K\alpha}{
          \varphi(m)\e\p{-m\frac h\alpha}
        }
      }
\end{equation*}
and
\begin{equation*}
      S_2(\alpha,h)
    =
      \sup_{
            f(A)<x\leq f(2A)
      }
      \abs{\sum_{x<m\leq x+K\alpha}\e\p{-m\frac h\alpha}}
.
\end{equation*}
The four summands in~(\ref{eqn:reduction_applied}) are arranged according to their importance.
We estimate them in the order of increasing importance, the treatment of the fourth term being trivial:
\begin{equation}\label{eqn:thm_first_step}
      \int_{f'(A)}^{f'(2A)}{
        O(1)
      }
      \ud\alpha
    \ll
      f'(2A)-f'(A)
.
\end{equation}
To estimate the third term, it is sufficient to consider the sum over $1\leq h\leq H$, since
$S_2(\alpha,-h)=S_2(\alpha,h)$.
We interchange the integration and the summation and substitute
$\theta=-\frac h\alpha$ to obtain
\begin{multline*}
    \int_{f'(A)}^{f'(2A)}{
      \frac 1H
      \sum_{1\leq h\leq H}{
        S_2(\alpha,h)
      }
    }
    \ud \alpha
\\
  \ll
    \frac 1{H}
    \sum_{1\leq h\leq H}{
      h
      \int_{-\frac h{f'(A)}}^{-\frac h{f'(2A)}}{
        \frac 1{\theta^2}
        \min\left\{
          f'(2A)K+1,\Abs{\theta}^{-1}
        \right\}
      }
      \ud \theta
    }
.
\end{multline*}
We note some simple estimates before applying Lemma~\ref{lem_integral_1}.
We have $0<f'(1)\leq f'(2A)\ll A^{\delta}$
for some $\delta\geq 0$ since $f'$ is monotone and by~(\ref{eqn:df_x}), and therefore
$f'(2A)K+1\ll A^{\delta+1}$.
By~(\ref{eqn:df_quotient}) we have
$0<-\frac 1\theta\leq\frac{f'(2A)}{h}\ll\frac{f'(A)}h$
for all $\theta$ under consideration.
Moreover, the length of the integration range
is
$
    \frac h{f'(A)}-\frac h{f'(2A)}
  \leq
    \frac h{f'(A)}
$
and finally from~(\ref{eqn:df_d2f}) and~(\ref{eqn:mvt_mon_2}) it follows that
$f'(A)\ll (f'(2A)-f'(A))\log A$.
Hence Lemma~\ref{lem_integral_1} gives
\begin{multline}\label{eqn:thm_second_step}
    \int_{f'(A)}^{f'(2A)}{
      \frac 1H
      \sum_{1\leq h\leq H}{
        S_2(\alpha,h)
      }
    }
    \ud\alpha
\\
  \ll
    f'(A)
    \frac 1H
    \sum_{1\leq h\leq H}{
      \frac {f'(A)}h
      \p{
          \frac h{f'(A)}
        +
          1
      }
      \p{1+\log A^{\delta+1} }
    }
\\
  \ll
      f'(A)
      \p{
          1
        +
          \frac{f'(A)\log H}{H}
      }
      \log A
\\
  \ll
      \bigl(f'(2A)-f'(A)\bigr)(\log A)^2
      \p{
        1+\frac{f'(A)\log H}{H}
      }
.
\end{multline}
The contribution of the second term in~(\ref{eqn:reduction_applied}) is
easily determined: the sum occurring in the definition of $S_2$ comprises not more than
$f'(2A)K+1\ll f'(A)K$ summands, therefore
\begin{equation}\label{eqn:thm_third_step}
    \int_{f'(A)}^{f'(2A)}\frac 1H S_2(\alpha,0)\ud \alpha
  \ll
    (f'(2A)-f'(A))K\frac{f'(A)}H
.
\end{equation}
Now we turn to the the treatment of the main term in~(\ref{eqn:reduction_applied}).
We concentrate on the case that $h>0$.
We exchange the integral and the sum and apply the substitution
$-\frac h\alpha=\theta$.
The factor $\min\left\{\frac 1\alpha,\frac 1h\right\}$ then transforms into
$\min\left\{-\frac 1\theta,\frac 1{\theta^2}\right\}$, which
is
$\ll \frac{f'(A)}h\min\left\{1,\frac{f'(A)}h\right\}$ by~(\ref{eqn:df_quotient}).
We obtain
\begin{multline*}
    \int_{f'(A)}^{f'(2A)}{
      \sum_{1\leq h\leq H}{
        \min\left\{\frac 1\alpha,\frac 1h\right\}
        S_1(\alpha,h)
      }
    }
    \ud\alpha
\\
  \ll
    f'(A)
    \sum_{1\leq h\leq H}{
      \frac 1h
      \min\left\{1,\frac{f'(A)}h\right\}
      \int_{-\frac h{f'(A)}}^{-\frac h{f'(2A)}}{
        S_1\p{-\frac h\theta,h}
      }
      \ud \theta
    }
\end{multline*}
and to estimate the integral we use Lemma~\ref{lem_vinogradov}:
\begin{multline*}
      \int_{-\frac h{f'(A)}}^{-\frac h{f'(2A)}}{
        S_1\p{-\frac h\theta,h}
      }
      \ud \theta
    \ll
      \int_0^1
        \min\left\{
            f'(2A)K+1
          ,
            \Abs{\xi}^{-1}
        \right\}
\\
    \times
      \int_{\xi-\frac h{f'(A)}}^{\xi-\frac h{f'(2A)}}
        \sup_{f(A)<x\leq f(2A)}
        \abs{
          \sum_{x<m\leq x+f'(2A)K}
          \varphi(m)
          \e\p{m\theta}
        }
      \ud \theta
      \ud \xi
.
\end{multline*}
The length of integration of the inner integral is bounded trivially by $\frac h{f'(A)}$ and the integrand is $1$-periodic,
so that we may replace this integral, using the definition~(\ref{eqn:expsum_integral}) of $J$, by the upper bound
\[
\p{\frac h{f'(A)}+1}
f'(2A)K\,
J\bigl(A,f'(2A)K\bigr)
,
\]
which is independent of $\xi$.
We use the estimate $f'(2A)K+1\ll A^{\delta+1}$ which we mentioned before and
Lemma~\ref{lem_integral_1} to obtain
\[
      \int_0^1
        \min\left\{
            f'(2A)K+1
          ,
            \Abs{\xi}^{-1}
        \right\}
      \ud \xi
    \ll
      \log A
.
\]
Splitting the summation over $h$ at $f'(A)$ we get
\begin{multline*}
    \sum_{1\leq h\leq H}{
      \frac 1h
      \min\left\{1,\frac{f'(A)}{h}\right\}
    }
    \p{
        \frac h{f'(A)}
      +
        1
    }
\\
  \ll
      \sum_{1\leq h\leq f'(A)}
        \frac 1h
    +
      \sum_{f'(A)<h\leq H}
        \frac 1h
        \frac {f'(A)}{h}
        \frac h{f'(A)}
  \ll
    \sum_{1\leq h\leq H}
      \frac 1h
  \ll
    \log H
.
\end{multline*}
Collecting the terms and using the estimate $f'(2A)\ll \bigl(f'(2A)-f'(A)\bigr)\log A$, which follows from Lemma~\ref{lem_properties_f}, we get
\begin{multline}\label{eqn:thm_fourth_step}
    \int_{f'(A)}^{f'(2A)}{
      \sum_{1\leq h\leq H}{
        \min\left\{\frac 1\alpha,\frac 1h\right\}
        S_1(\alpha,h)
      }
    }
    \ud \alpha
\\
  \ll
    f'(A)
    \bigl(f'(2A)-f'(A)\bigr)K
    (\log A)^2\log H\,
    J\bigl(A,f'(2A)K\bigr)
.
\end{multline}
By analogous reasoning
the sum over $-H\leq h\leq -1$ can be estimated by the same expression.
We choose
\[H=z\quad\text{and}\quad K=\frac{z}{f'(2A)}.\]
By the restrictions $\max\{2,f'(2A)\}\leq z\leq A\,f'(A)$ it easily follows that
$1\leq K\leq A$ and that $H\geq 2$, therefore this is an admissible choice.
Note also that $\log H\ll \log A$ by~(\ref{eqn:df_x}).
We combine
~(\ref{eqn:reduction_applied}),
~(\ref{eqn:thm_first_step}),
~(\ref{eqn:thm_second_step}),
~(\ref{eqn:thm_third_step}) and
~(\ref{eqn:thm_fourth_step})
to get the estimate
\begin{equation*}\label{eqn:connection}
    I\left(A,\frac{z}{f'(2A)}\right)
  \ll
      \frac {f'(A)(\log A)^3}{z}
    +
      f'(A)(\log A)^3J(A,z)
.
\end{equation*}
Applying Proposition~\ref{prp_1} we see that the left hand side of
~(\ref{eqn:expsum_conclusion}) is bounded by a constant times
\begin{equation*}\label{eqn:temp_error_term}
  \frac{f''(A)}{f'(A)^2}z^2
+
  \frac{f'(A)(\log A)^3}z
+
  f'(A)(\log A)^3J(A,z).
\end{equation*}
By~(\ref{eqn:J_lower_bound}) the second term in this expression is dominated by the third,
which completes the proof.
\end{proof}
\subsection*{Acknowledgements}
The author is grateful to Michael Drmota for helpful discussions
and to Jo\"el Rivat for suggesting numerous improvements of this article.
Moreover, the author would like to thank the anonymous referee for
carefully reading the paper.

The author was supported by the Austrian Science Fund (FWF), grants P22388 and P24725,
and by the Agence Nationale de la Recherche, grant ANR-10-BLAN 0103 MUNUM.
Moreover, the author acknowledges support by project F5502-N26 (FWF), which is a part of the Special Research Program ``Quasi Monte Carlo Methods: Theory and Applications''.




\bibliographystyle{siam}
\bibliography{bibliography}
\end{document}